\documentclass{birkmult}

\usepackage[latin1]{inputenc}
\usepackage[english]{babel}
\usepackage{amsmath,amsthm,amssymb,latexsym}
\usepackage{graphicx,color}
\usepackage{tikz}
\usepackage{comment}
\usepackage[T1]{fontenc}
\usepackage{times,mathptmx}
\usepackage{fancyhdr}
\usepackage{hyperref}
\usepackage{wrapfig}
\usepackage{array}

\makeatletter

\newcommand{\IN}{\mathbb{N}}
\newcommand{\IZ}{\mathbb{Z}}

\newcommand{\IR}{\mathbb{R}}

\renewcommand{\H}{\mathcal{H}}
\newcommand{\C}{\mathcal{C}}

\newcommand{\N}{{{\mathcal N}}}

\newcommand{\E}{{{\mathcal E}}}
\newcommand{\A}{{{\mathcal A}}}
\newcommand{\grad}{\nabla}
\newcommand{\abs}[1]{\mathopen\vert#1\mathclose\vert}
\newcommand{\bigabs}[1]{\bigl| #1\bigr|}

\newcommand{\norm}[1]{\mathopen\Vert#1\mathclose\Vert}

\newcommand{\inner}[2]{\langle #1 \vert #2\rangle}
\newcommand{\biginner}[2]{\bigl\langle #1 \,{\bigm|}\, #2\bigr\rangle}
\newcommand{\intd}{\,{\mathrm d}}
\renewcommand{\phi}{\varphi}
\renewcommand{\epsilon}{\varepsilon}
\@ifundefined{leqslant}{}{\let\le=\leqslant  \let\leq=\leqslant}
\@ifundefined{geqslant}{}{\let\ge=\geqslant  \let\geq=\geqslant}
\@ifundefined{varnothing}{}{\let\emptyset=\varnothing}
\newcommand{\intervaloo}[1]{(#1)}
\newcommand{\intervalco}[1]{[#1)}
\newcommand{\intervaloc}[1]{(#1]}
\newcommand{\intervalcc}[1]{[#1]}
\newcommand{\bigintervaloo}[1]{\bigl(#1\bigr)}
\newcommand{\bigintervalco}[1]{\bigl[#1\bigr)}
\renewcommand{\subset}{\subseteq}

\newcommand{\limplies}{\Rightarrow}
\DeclareMathOperator{\Ran}{Ran}
\DeclareMathOperator{\dist}{dist}
\DeclareMathOperator{\spanned}{\overline{span}}

\theoremstyle{plain}
\newtheorem{Prop}{Proposition}[section]
\newtheorem{Thm}[Prop]{Theorem}
\newtheorem{Lem}[Prop]{Lemma}

\newtheorem{algo}[Prop]{Algorithm}
\theoremstyle{definition}

\newtheorem{Def}[Prop]{Definition}

\theoremstyle{remark}
\newtheorem{Rem}[Prop]{Remark}

\renewcommand{\theenumi}{\roman{enumi}}
\renewcommand{\labelenumi}{(\theenumi)}

\definecolor{fixme}{RGB}{224,15,182}

\newcommand{\setlabel}[2]{%
  \def\@currentlabel{#2}%
  \label{#1}%
}

\definecolor{mygreen}{rgb}{0.2,0.47,0.2}
\definecolor{myblue}{rgb}{0.23,0.33,0.95}
\definecolor{myred}{rgb}{0.67,0.03,0.06}

\@ifundefined{hypersetup}{}{%
  \hypersetup{%
    pdfauthor={Grumiau Christopher, Troestler Christophe},
    pdftitle={Convergence of a mountain pass type algorithm for
      strongly indefinite problems and systems},
    pdfsubject={Partial differential equations},
    pdfkeywords={Mountain pass algorithm, minimax, steepest descent method,
      Schr\"odinger equation, spectral gap,
      strongly indefinite functional, ground state solutions,
      Nehari manifold, systems},
    colorlinks=true,
    urlcolor=myblue,
    citecolor=myblue,
    linkcolor=myred,
    pdfpagemode=UseNone,
    pdfstartview=FitH,
  }%
}

\makeatother

\begin{document}

%
\title[Mountain pass algorithm for indefinite problems and systems]{%
  Convergence of a mountain pass type algorithm for
  strongly indefinite problems and systems}
\author[Ch.~Grumiau, Ch.~Troestler]{Grumiau Christopher, Troestler Christophe}

\address{%
 Institut Complexys\\
 D{\'e}partement de Math{\'e}matique\\
 Universit{\'e} de Mons, \\
 20, Place du Parc\\
 B-7000 Mons\\
 Belgium
}

\email{christopher.grumiau@umons.ac.be}
\email{christophe.troestler@umons.ac.be}

\thanks{The authors are partially supported by a grant from the
  National Bank of Belgium and by the program
``Qualitative study of
 solutions of variational elliptic partial differerential equations. Symmetries,
bifurcations, singularities, multiplicity and numerics'' of the FNRS, project
2.4.550.10.F of the Fonds de la Recherche Fondamentale Collective.}




\begin{abstract}
  For  a functional $\E$ and a peak selection that picks up a global
  maximum of $\E$ on varying cones, we study the convergence up to
  a subsequence to a critical point of the sequence generated by a
  mountain pass type algorithm. Moreover,
  by carefully choosing stepsizes, we establish the convergence of the whole
  sequence under a ``localization'' assumption on the critical point.
  We illustrate our results with two problems: an indefinite
  Schrödinger equation
  and a superlinear Schrödinger system.
\end{abstract}

\maketitle

\section{Introduction}

Let us consider $\H$ a Hilbert space with inner product
$\inner{\cdot}{\cdot}$ and norm $\norm{\cdot}$, and a functional $\E
\in \C^1(\H; \IR)$.
In this work, we develop a provably convergent ``general'' mountain pass type
algorithm to approximate saddle points of $\E$,
with a Morse index possibly larger than one.
The pioneer work in this direction is due to  Y.~S.~Choi and
P.~J.~McKenna~\cite{mckenna} who proposed a
constrained steepest descent method to compute saddle points with
one ``descent direction'' (such as a Mountain Pass solution).
A proof of convergence of a variant of that algorithm
was later given by Y.~Li and J.~Zhou in~\cite{zhou1,zhou2}.
To briefly describe it,
let us fix a closed subspace $E$ of $\H$ and $\phi$
a continuous $E^\perp$-peak selection,
i.e.\ $\phi(u)$ is the location of a maximum of
$\E$ on $E \oplus \IR^+ u := \{ e + ty \mid e \in E,\  t \ge 0\}$
for any $u\in \H\setminus E$ and $\phi$ is constant on $E \oplus \IR^+ u$.
As it will be convenient in the rest of the paper that $\phi$ is
not solely defined on a unit sphere, we present here a slightly
different version~\cite{tt}.

\begin{algo}[Mountain Pass Algorithm]
  \label{mpa}
  \begin{enumerate}
  \item Choose $u_0\in\Ran \phi$, $\varepsilon >0$ and $n\leftarrow 0$;
  \item if $\norm{\nabla \E(u_n)}\leq \varepsilon $ then  stop; \\
    else compute
    \begin{equation*}
      u_{n+1}
      = \phi\left(u_n - s_n
        \frac{\nabla \E(u_n)}{\norm{\nabla \E(u_n)}}\right),
    \end{equation*}
    for some $s_n\in S(u_n) \subset \intervaloo{0, +\infty}$
    where $S(u_n)$ is a set of ``admissible stepsizes''
    chosen so that at least the following inequality holds:
    \begin{equation*}
      \E(u_{n+1})-\E(u_n)
      < - \tfrac{1}{2} s_n \norm{\nabla \E(u_n)};
    \end{equation*}
  \item let $n\leftarrow n+1$ and go to step $2$.
  \end{enumerate}
\end{algo}

Y.~Li and J.~Zhou proved that $(u_n)$ converges
to a nontrivial critical point of $\E$ up
to a subsequence.
The proof of convergence is performed in the space $\H$
to ensure that the rate of
convergence for the discretized problem does not deteriorate
when the approximating subspace becomes finer.
The original goal of the authors for
introducing $E$  was to
try to obtain multiple critical points by taking $E$
as the linear subspace generated by previously found solutions
which the algorithm must
try to avoid.  The proof is performed in two steps.
First, they show
that $s_n$  exists and that $\E$ decreases along $(u_n)_{n\in\IN}$.
This step relies on the following deformation lemma.
\begin{Lem}
  \label{Udeflemma}
  If $\phi$ is continuous, $u_0\in \Ran \phi$, $u_0\notin E$, $\nabla
  \E(u_0)\neq 0$,
  then there exists $s_0>0$ such that
  \begin{equation*}
    \forall s \in \intervaloc{0, s_0},\quad
    \E\bigl(\phi(u_{s})\bigr) - \E(u_0)
    < - \tfrac{1}{2} s \norm{\nabla \E(u_0)},
  \end{equation*}
  where
  \begin{equation*}
    u_{s}
    := u_0 - s \frac{\nabla \E(u_0)}{\norm{\nabla \E(u_0)}} .
  \end{equation*}
\end{Lem}
The second step consists in proving,
under some traditional assumptions on $\phi$, that
a subsequence of $(u_n)$ converges.  For this, it is essential to show
that the stepsize $s_n$ controls the distance between $u_n$ and $u_{n+1}$
and that $s_n$ is chosen in such a way that it is close to $0$
only when ``mandated'' by the functional.  Let us remark that
the choice of $\phi$ is very
sensitive. Indeed,  to seek  sign-changing
critical points, the modified
mountain pass algorithm was introduced by
J.~M.~Neuberger~\cite{Neuberger97} (see also~\cite{cdn}). He
considers  algorithm~\ref{mpa} above and only modifies the projection $\phi$
into a ``sign-changing peak selection'' $\phi_N$ which is a
map defined from
the set of sign-changing functions of
$\H\setminus\{0\}$ to $\H\setminus\{0\}$ such that, for any
$u$, $\E\bigl(\phi_N(u)\bigr)>0$ and $\phi_N(u)$ is a maximum of $\E$ on
$\IR^+ u^+ \oplus \IR^+ u^-$ where $u^+(x):= \max\{0,u(x)\}$ and
$u^-(x):=\min\{0, u(x)\}$.  Although in practice it appears to converge to a
nontrivial sign-changing critical point, its
convergence has yet to be formally proved.

In this paper, $\phi(u)$ is allowed to
pick up a maximum point of $\E$ in an
abstract cone $C_u$
and we are interested in giving
assumptions on $C_u$ which imply the convergence of the
mountain pass algorithm. This work is partly motivated by the
article~\cite{noris} where
the authors define the notion of ``natural
constraints'' to seek nontrivial
critical points of functionals.
Let us first make precise the peak selection $\phi$ that we use.
We write
$\operatorname{int}C$ for the interior of $C$ relative to $\spanned C$,
the smaller closed subspace containing $C$,
for the topology induced by~$\H$.
\begin{Def}
  \label{defpeak}
  Let $\A$ be an open subset of~$\H$.
  We say that $\phi$ is a \emph{peak selection for $(C_u)_{u \in \A}$}
  if $\phi$ is a map from $\A$ to $\A$ such
  that, for all $u \in \A$,
  \begin{enumerate}
  \item $C_u$ is a closed cone pointed at~$0$;
  \item $\phi(u) \in \operatorname{int} C_u$;
  \item\label{phi-uniqueness} for any $v\in \operatorname{int}C_u$,
    $\phi(v)= \phi(u)$;
  \item $\phi(u)$ is a global maximum point of $\E$ on $C_u$.
  \end{enumerate}
\end{Def}

Note that properties (ii) and~(iii) imply that $\phi(\phi(u)) =
\phi(u)$.
We write that $d\perp C_u$ if and only
if $d\perp \spanned C_u $ for the inner product~$\inner{\cdot}{\cdot}$.
In section~\ref{conmpa}, we assume that
$\phi$ is continuous and
that  $(C_u)$ verifies the following conditions:
\begin{align*}
  &\forall u\in \A, \ u\in C_u  \quad\text{and }
  \tag{$AC_1$}\label{Ac1}\\
  \tag{$AC_2$}\label{Ac2}
  \begin{split}
    &\exists \gamma \in \bigintervaloo{0,\frac{\pi}{2}}, \
    \exists \delta \in (0,1),\  \forall u_0\in \Ran\phi,\
    \exists r>0,\
    \forall \tilde{u}_0\in \Ran\phi \cap B(u_0,r),\   \\
    &\quad
    \forall d \in B(0,r),\ d \perp C_{\tilde{u}_0},\quad
    C_{\tilde{u}_0 + d} \cap B(u_0, r)\subset C_{\tilde{u}_0}
    + [1-\delta, 1+\delta] \,  A_\gamma(d)
  \end{split}
\end{align*}
where $A_\gamma(d) := \{ d' \mid \norm{d'} = \norm{d}$ and $\angle
(d',d)\le \gamma\}$ and
$\angle (d,d') := \arccos \bigl(\frac{\inner{d}{d'}}{\norm{d}\norm{d'}}\bigr)$
denotes the angle between two
non-zero vectors $d$ and $d'$ (we set $A_\gamma(0) := \{0\}$).
This assumption, which speaks about the behavior of the cones under
small deformations, is essential to prove a deformation Lemma in this
generalized setting
(see Lemma~\ref{deflemma}).
This lemma ensures the non-emptiness of the
set $S(u_0)$ of admissible stepsizes at $u_0$ which we now define.
For any $u_0\in \Ran \phi$ such that $\nabla \E(u_0)\ne 0$, we set
\begin{equation*}
  S^{*}(u_0) := \Bigl\{ s>0 \Bigm|  u_s := u_0 -
  s\frac{\nabla\E(u_0)}{\norm{\nabla\E(u_0)}} \in \A
  \text{ and }
  \E\bigl(\phi(u_s)\bigr) - \E(u_0)
  < -\alpha s\norm{\nabla \E(u_0)}  \Bigr\}
\end{equation*}
for some value $\alpha >0$ given  by  Lemma~\ref{deflemma}  and  we require that $s\in S(u_0):= S^{*}(u_0)\cap
\bigintervalco{\frac{1}{2}\sup  S^{*}(u_0), +\infty}$.  Other definitions of
admissible stepsizes are possible provided they imply a local
uniformity in the sense that $s_n \in S(u_n)$ forces the stepsize $s_n$ not
to be small when the gradient is not (see Lemma~\ref{convlemma1}).
The definition given above
draws its inspiration from a
paper~\cite{tt} written by N.~Tacheny and C.~Troestler.

\enlargethispage{8mm}

To obtain the convergence of $(u_n)$ up to a subsequence
(see Theorem~\ref{convthm}), we
unfortunately need to replace~\eqref{Ac2} with the following  stronger
assumption:
\begin{equation}
  \tag{$AC_3$}\label{Ac-xi}
  \left.\begin{minipage}{0.83\linewidth}
      there exists a closed subspace $E\subset \H$ (possibly
      infinite dimensional) and a family of
      $\C^1$-vector fields $\xi_i :\A \to E^{\perp}$,
      $i=1,\dots,k$, for some $k\in\IN$, such that
      for all $u \in \A$ and $i \in \{1,\dotsc,k\}$,
      \begin{enumerate}
      \item the family $\bigl(\xi_i(u)\bigr)_{i=1}^k$ is orthonormal;
        \vspace{0.5ex}
      \item $\forall v \in V_u,\ \xi'_i(u)[v] \in V_u$,
        where $V_u := \spanned \{\xi_1(u), \dotsc, \xi_k(u)\}$;
        \vspace{0.5ex}
      \item $\forall v\in \operatorname{int} C_u\cap\A$,
        $\xi_i(v) = \xi_i(u)$;
        \vspace{0.5ex}
      \item $\inner{u}{\xi_i(u)} \ne 0$;
        \vspace{0.5ex}
      \item $\exists r > 0$, $\xi'_i$ is bounded on
        $\{ u \mid \dist(u, \Ran \phi) < r \} \cap \A$.
      \end{enumerate}
      For $u \in \A$, the cone $C_u$ is defined as
      $C_u := E \oplus \{ \sum_i t_i \xi_i(u) \mid
      t_i \ge 0 \text{ for all } i \}$.
    \end{minipage} \ \right\}
\end{equation}
Here, the notation $\dist(u,\partial\A)$ stands for
$\inf\{ \norm{u - v} \mid v \in \partial\A \}$.
Let us remark conditions~\eqref{Ac-xi} (i), (ii), (iv) and (v)
are already present (albeit somehow implicitly for~(iv))
in~\cite{noris} in the context of trivial $\C^1$-subbundles intead of
cones.
The additional condition~(iii)
is equivalent to
$\forall v \in \operatorname{int} C_u\cap\A,\ C_v = C_u$.
This condition is rather natural to require in view of
property~(iii) of the definition of peak selection.
This ``finite presentation'' of the
cones is used  in Lemma~\ref{convlemma2-xi} to ensure that the stepsize $s_n$
controls the distance between $u_{n+1}$ and~$u_n$.

As a particular case of~\eqref{Ac-xi}, let us mention that we can
work with a family of continuous linear projectors
(see Proposition~\ref{conv-P_i}).  This
case is an abstract formulation of the setting of~\cite{sym2,sym1}
where the convergence (up
to a subsequence) of a mountain pass type algorithm for systems
has been announced.

In Section~\ref{convmpa2}, we are interested in the convergence of the whole
sequence generated by the Mountain Pass Algorithm.  To that aim, we
need to refine the definition of $S^*$  in order to control
$\E\bigl(\phi(u_n - s \frac{\grad u_n} {\norm{\grad u_n}})\bigr)$
for any $0<s<s_n$.

In Section~\ref{applmpa}, we illustrate our method with  two semi-linear
problems. The first application takes its inspiration from a paper due to
A.~Szulkin and T.~Weth~\cite{szulkin} in which the authors study the following
Schrödinger problem
\begin{equation}
\label{eqSzulkin}
  \left\{
    \begin{aligned}
      -\Delta u(x) + V(x) u(x)&= \abs{u(x)}^{p-2}u(x),&&\ x\in\Omega, \\
      u(x)&=0,&&\ x \in\partial\Omega ,
    \end{aligned}
  \right.
\end{equation}
where $V:\Omega\to \IR$ is such that $0$ is in a spectral gap
of $-\Delta + V$ and $2<p<2^*:=\frac{2N}{N-2}$
($+\infty$ when $N=2$).  They are interested
in the existence of
non-zero solutions on an open bounded domain $\Omega\subset \IR^N$ or
on $\Omega = \IR^N$ (in the latter case, $V$ is assumed to be $1$-periodic in
each $x_i$, $i=1,\dotsc,N$).   Solutions to this equation are
critical points of the indefinite functional
\begin{equation}
  \label{functional-indefinite}
  \E: \H\to \IR : u\mapsto \frac{1}{2}\int_\Omega\left(
    \abs{\nabla u(x)}^2 + V(x) u(x)^2\right) \intd x
  -\frac{1}{p}\int_\Omega \abs{u (x)}^p\intd x,
\end{equation}
where $\H= H^1_0 (\Omega)$. The first proof of the existence of
non-zero critical points for $\E$ when $-\Delta + V$ is not positive
definite and $\Omega$ is an open bounded domain is due to
P.~H.\ Rabinowitz~\cite{rabin}.  Recently,
A.~Szulkin and T.~Weth proposed an alternative method~\cite{szulkin}
that also makes easier to deal with
the lack of compactness that occurs when $\Omega = \IR^N$.  Denoting $E$
the negative eigenspace of $-\Delta + V$, they introduce
the following nonlinear map
\begin{equation*}
  \phi: \H\setminus E\to \H: u\mapsto \phi(u)
\end{equation*}
where $\phi(u)$ is the point at which
$\E$ reaches its maximum value on $E \oplus \IR^+ u$.  They prove that
minimizing $\E$ on $\Ran \phi = \bigl\{ u\in \H\setminus
E \bigm| \partial\E(u)[v]=0$
for $v=u$ and any $v\in E \bigr\}$ yields a non-zero
solution with least energy.
Notice that, here, $E$  is used to deal with
the indefiniteness of the problem
and not to compute multiple critical points as in
the papers of J.~Zhou \& al.~\cite{zhou1,zhou2}.  We will
prove that our algorithm converges for this problem.  The numerical
solutions that we obtain lead to some conjectures on
the symmetries of ground state solutions.

The second application is based on a paper by
B.~Noris and G.~Verzini \cite{noris}. The
authors study the superlinear Schrödinger system
\begin{equation}
  \label{eqNoris-general}
  \left\{\!
    \begin{aligned}
      -\Delta u_i(x) &= \partial_iF\bigl(u_1(x),\dots,u_k(x)\bigr),
      &&\ x\in\Omega, \\
      u_i(x)&=0,&&\ x \in\partial\Omega ,
    \end{aligned}
  \right.
  \qquad i=1,\dots,k,
\end{equation}
where $k \in \IN$. They require that  $\Omega\subset \IR^N$ is
a bounded smooth
domain and $F\in \C^2(\IR^k; \IR)$.
Note that the system $-\Delta u_i =\mu_i
u_i^3 + u_i \sum_{j\neq i} \beta_{i,j} u_j^2$ where $\mu_i>0$
and $\beta_{i,j} = \beta_{j,i}$ is a particular case
of~\eqref{eqNoris-general}.  Such type of nonlinearities have been studied
due to their applications to nonlinear optics and to Bose-Einstein
condensation (see~\cite{conti,terracini2,dancer,terracini}).
Solutions to~\eqref{eqNoris-general} are
critical points of the functional
\begin{equation}
  \label{functional-system}
  \E: \H\to \IR : u = (u_1,\dotsc, u_k) \mapsto
  \frac{1}{2}\int_\Omega\abs{\nabla u(x)}^2
  \intd x - \int_\Omega F(u)\intd x,
\end{equation}
where $\H = H^1_0 (\Omega; \IR^k)$.
As already mentioned, B.~Noris and G.~Verzini~\cite{noris}
propose a general method of ``natural constraints''.
Applied to the above problem, it goes as follows.
Denote $\A := \{u\in \H \mid
u_i\neq 0 \text{ for every } i\}$.  To find a solution
$u=(u_1,\dots, u_k)$ of~\eqref{eqNoris-general}
with $u_i \ne 0$ for all $i=1,\dots,k$,
they minimize $\E$ on the constraint $\N := \bigl\{u\in \A
\bigm| \forall  i=1,\dots,k,\ \int_{\Omega}
\abs{\nabla u_i}^2\intd x = \int_{\Omega} \partial_i F(u)u_i\intd x
\bigr\}$.   We will show that, under their assumptions,
$\N =\Ran \phi$ with $\phi$ being the peak selection
\begin{equation*}
  \phi : \A \to \A : u \mapsto
  \operatorname{argmax} \bigl\{ \E(t_1 u_1,\dots, t_k u_k) \bigm|
  t_1>0, \dotsc, t_k > 0 \bigr\} .
\end{equation*}
Again, we prove that our algorithm converges for this problem and
some numerical experiments are performed.

\section[Steepest descent method]{Steepest
  descent method on varying cones}
\label{conmpa}

\subsection{Uniform deformation lemma}

Let $\E : \H \to \IR$ be a $\C^1$-functional defined on a Hilbert space
$\H$ and $\A$ an open subset of~$\H$.  The following lemma is
instrumental in proving the convergence of the algorithm.

\begin{Lem}[Uniform deformation lemma]
  \label{deflemma}
  Let $\phi : \A \to \A$ be a peak selection for $(C_u)_{u \in \A}$ and
  $u_0\in \Ran\phi$ be such that $\nabla \E(u_0) \ne 0$.
  Assume that $\phi$ is continuous at~$u_0$
  and that \eqref{Ac1}--\eqref{Ac2} hold.
  Then there exist $s_0>0$ and $r_0 >0$ such that,
  for any $s \in \intervaloc{0, s_0}$ and
  $\tilde{u}_0\in B(u_0, r_0)\cap \Ran \phi$, one has
  \begin{itemize}
  \item $\nabla \E(\tilde{u}_0)\neq 0$,
  \item $\tilde{u}_s\in \A$ where $\tilde{u}_s:= \tilde{u}_0
    -s\frac{\nabla \E(\tilde{u}_0)}{\norm{\nabla \E(\tilde{u}_0)}}$ and
  \item there exists some $\alpha > 0$ solely depending on $\gamma$
    and $\delta$ given in assumption~\eqref{Ac2} such that
    \begin{equation}
      \label{step}
      \E\bigl(\phi(\tilde{u}_s)\bigr) - \E(\tilde{u}_0)
      < -\alpha s\norm{\nabla \E(\tilde{u}_0)}.
    \end{equation}
  \end{itemize}
\end{Lem}

\begin{proof}
  Let $u_0 \in \Ran\phi \subset \A$ and
  let us consider $\gamma$, $\delta$
  and $r$ given by the assumption~\eqref{Ac2} for~$u_0$.
  Since $\A$ is open, there exists
  $\varepsilon_1 >0$ such that for any $u\in B(u_0, \varepsilon_1)$ and
  $v\in B(u,\varepsilon_1)$,
  one has $u,v\in \A$, $\nabla\E(u)\ne 0$,
  $\nabla\E(v)\ne 0$ and $u,v\in B(u_0, r)$.

  For any $u\in B(u_0, \varepsilon_1)$, let $d_u := -{\nabla \E
  (u)}/{\norm{\nabla \E(u)}}$.  Then
  \begin{equation*}
    \forall u\in B(u_0, \varepsilon_1),\
    \forall d\in A_\gamma(d_u),\quad
    \inner{\nabla\E(u)}{d} \le -\cos\gamma\, \norm{\nabla\E(u)}.
  \end{equation*}
  Let $\tilde\gamma := \frac{1}{2}\cos\gamma > 0$.
  Taking $\varepsilon_1$ smaller if necessary, we may assume that
  \begin{equation*}
    \forall u, v \in B(u_0, \varepsilon_1),\
    \forall d\in A_\gamma(d_u),\quad
    \inner{\nabla\E(v)}{d} < -\tilde\gamma \norm{\nabla\E(u)}.
  \end{equation*}
  Thus, on one hand, there exists $\varepsilon_2>0$ such that,
  for any $u\in
  B(u_0, \varepsilon_2)$, $v \in B(u,\varepsilon_2)$, $d\in
  A_\gamma(d_u)$ and $\sigma \in \intervaloo{0, \varepsilon_2}$,
  \begin{equation*}
    \inner{\nabla\E(v + \sigma d)}{d}
    < - \tilde\gamma \norm{\nabla\E(u)}.
  \end{equation*}
  For any  $\tilde{u}_0\in B(u_0, \varepsilon_2)\cap
  \Ran \phi$, $v\in C_{\tilde{u}_0}\cap B(\tilde{u}_0,
  \varepsilon_2)$,  $d\in A_\gamma(d_{\tilde{u}_0})$ and
  $\sigma < \varepsilon_2$, the mean value theorem implies there
  exists a $\tilde\sigma \in \intervaloo{0,\sigma}$ such that
  \begin{align}
    \E(v  + \sigma d) - \E(\tilde{u}_0)
    &\le \E(v+ \sigma d) -\E(v)  \label{eq:use max E}\\
    &= \biginner{\nabla \E( v+ \tilde\sigma d)}{\sigma d} \notag\\
    &< -\tilde\gamma \sigma \norm{\nabla \E(\tilde{u}_0)},
    \label{eqlem}
  \end{align}
  where the first inequality results from $v\in C_{\tilde{u}_0}$
  and $\tilde{u}_0 = \phi(\tilde u_0)$ is a global maximum of $\E$
  on~$C_{\tilde{u}_0}$.

  On the other hand, by the continuity of $\phi$ at~$u_0$, there exist
  $s_0 \in \intervaloo{0,r}$
  and $\varepsilon_3 \in \bigintervaloo{0,
    \min\{r, \tfrac{1}{3}\varepsilon_2\}}$ such that,
  for any $\tilde{u}_0\in B(u_0,\varepsilon_3)$
  and $s \in \intervalcc{0,s_0}$, one has
  $\phi(\tilde{u}_0 + s d_{\tilde{u}_0})
  \in B(u_0,\min\{r, \tfrac{1}{3}\varepsilon_2\})$.
  Let $\tilde{u}_s := \tilde{u}_0 + s d_{\tilde{u}_0}$.
  If in addition $\tilde u_0 \in \Ran\phi$, one has
  $d_{\tilde{u}_0} \perp \spanned C_{\tilde{u}_0}$ (because $\tilde{u}_0 =
  \phi(\tilde{u}_0) \in \operatorname{int} C_{\tilde{u}_0}$ is a local
  maximum)
  and therefore one deduces from assumption~\eqref{Ac2} that
  \begin{equation*}
    \phi(\tilde{u}_s)\in C_{\tilde{u}_s} \cap B(u_0,r)
    \subset C_{\tilde{u}_0}
    + \intervalcc{1-\delta,1+\delta} A_{\gamma}(sd_{\tilde{u}_0}).
  \end{equation*}
  Thus, $\phi(\tilde{u}_s) = v_s + K_s s d^*_s$
  for some $v_s\in C_{\tilde{u}_0}$,
  $K_s\in \intervalcc{1-\delta, 1+\delta}$ and $d^*_s\in
  A_{\gamma}(d_{\tilde{u}_0})$. So, possibly taking $s_0$
  smaller, we get that $K_s s < \tfrac{1}{3}\varepsilon_2$ and
  $v_s = \phi(\tilde{u}_s) -
  K_s s d^*_s \in B(\tilde{u}_0,\varepsilon_2)$.
  Using equation~\eqref{eqlem}, we conclude that
  \begin{equation*}
    \E\bigl(\phi(\tilde{u}_s)\bigr) - \E(\tilde{u}_0)
    = \E(v_s + K_s s d^*_s) - \E(\tilde{u}_0)
    \le -\tilde\gamma\, (1 -\delta) s \norm{\nabla\E(\tilde{u}_0)}
  \end{equation*}
  for any $\tilde{u}_0\in B(u_0,\varepsilon_3)\cap \Ran \phi$
  and $s \in \intervaloc{0, s_0}$.
\end{proof}

\begin{Rem}
\label{uniform1}
\begin{itemize}
\item Equation~\eqref{eq:use max E} is the unique place we use  that
  $\phi(u)$ is a global maximum of $\E$ on $C_u$.  This assumption can be
  weakened by only requiring that the neighborhood on which $\phi(u)$
  achieves the maximum of $\E$ is locally uniform w.r.t.\ $u$:
  \begin{equation}
    \label{eq:local-max-phi}
    \forall u_0 \in \Ran\phi,\
    \exists \rho >0,\
    \forall u \in \Ran \phi \cap B(u_0, \rho),\quad
    \E\bigl(\phi(u)\bigr) = \max_{v \in C_u \cap B(u, \rho)} \E(v).
  \end{equation}
  This assumption allows the existence of multiple maximums points
  in~$C_u$.  It was not used in definition~\ref{defpeak} for
  simplicity but also because, in the examples of
  section~\ref{applmpa}, $\phi(u)$ is a maximum on the whole~$C_u$.
\item Let us also note that, if we are just interested in the
  inequality~\eqref{step} at $u_0$ (and not for all $\tilde u_0$ in a
  neighborhood of $u_0$), we only need to require that
  $\phi(u)$ is a local maximum of $\E$ on~$C_u$.
\item A careful reader may notice that we did not really use the fact
  that $C_u$ is a cone pointed at~$0$.  However, if $(C_u)$ was just a
  family of sets satisfying
  \eqref{Ac1}, \eqref{Ac2}, \eqref{eq:local-max-phi} and the fact that
  $\phi(u) \in \operatorname{int} C_u$ in a locally uniform way:
  \begin{equation}
    \label{eqrem}
    \forall u_0 \in \Ran\phi,\ \exists \rho > 0,\
    \forall u \in \Ran \phi \cap B(u_0, \rho),\quad
    B(u, \rho) \cap \spanned C_u \subset C_u \text{,}
  \end{equation}
  then the cone $\hat C_u$,
  defined as the closure of $\{ tv \mid t \ge 0$ and $v \in
  C_u\}$, also satisfies \eqref{Ac1}, \eqref{Ac2}, \eqref{eq:local-max-phi}
  and $\phi(u) \in \operatorname{int} \hat C_u$.  So very little is
  gained by not using cones, especially because they are the natural
  structures encountered in our examples.
\item As a consequence of the above deformation lemma, one can interpret
  $\Ran\phi$ as
  somewhat a natural constraint for $\E$ in the sense of~\cite{noris}.
  More precisely, it implies that
  if $u_0\in \Ran \phi$ is a local minimum of $\E$ on $\Ran
  \phi$ then $u_0$ is a critical point of $\E$ on the whole space~$\H$.
\end{itemize}
\end{Rem}

\subsection{Convergence up to a subsequence}

In this section, we first remark that
it is possible to
construct a sequence of stepsizes $s_n$ such that the energy $\E$ decreases
along the sequence $(u_n)_{n\in\IN}$ generated by
algorithm~\ref{mpa}. In the following,  without loss of
generality, we can assume that $\nabla \E(u_n)\neq 0$ for  any
$n\in\IN$ (otherwise the algorithm finds a critical point
in a finite number of steps).

\begin{Prop}
  \label{stability}
  If   $s_n >0$ verifies inequality~\eqref{step} given in
  Lem\-ma~\ref{deflemma} for any $n\in\IN$,
  then the functional $\E$ decreases along the sequence $(u_n)_{n\in\IN}$.
\end{Prop}

\begin{proof}
  As $\nabla \E(u_n)\neq 0$, $s_n$ is well-defined by Lemma~\ref{deflemma}. By
  construction, we have
  \begin{equation*}
    \E(u_{n+1})-\E(u_n)
    = \E\left(\phi\bigl(u_n-s_n\frac{\nabla \E(u_n)}{\norm{\nabla
          \E(u_n)}}\bigr)\right)-\E(u_n)
    < -\alpha s_n  \norm{\nabla \E(u_n)}<0.
  \end{equation*}
  So, $\E(u_{n+1})<\E(u_n)$.
\end{proof}

As explained in the introduction, we now consider
the sets $S^*(u_0)$ and $S(u_0)$. The set $S^*(u_0)$ is
not empty as soon as $u_0$ is
not a critical point of $\E$ (thanks to the deformation
lemma). Concerning the set $S(u_0)$, it is not-empty once $\E$ is bounded from
below on $\Ran \phi$, an assumption that we will later make (see
Theorem~\ref{convtheorem}).

\begin{Lem}
  \label{convlemma1}
  If
  $u_0\in \Ran \phi$, $\nabla \E(u_0)\neq 0$ and  $\phi$ is continuous at
  $u_0$, then there exists an open
  neighborhood $V$ of $u_0$ and $s^*>0$ such that $S(u)\subset [s^*,+\infty)$
  for any $u\in V\cap \Ran \phi$.
\end{Lem}
\begin{proof}
  By the uniform deformation Lemma~\ref{deflemma}, there exists $s_0>0$ and
  $r_0>0$ such that, for any $0<s\leq s_0$ and $u\in B(u_0, r_0)\cap
  \Ran\phi$, we have $u_s := u-s
  \frac{\nabla \E(u)}{\norm{\nabla\E(u)}}\in \A$, $\grad\E(u)\neq 0$ and
  \begin{equation}
    \label{step2}
    \E\bigl( \phi(u_s) \bigr)-\E(u)
    < -\alpha s \norm{\nabla \E(u)}.
  \end{equation}
  In particular, for any $u\in B(u_0, r_0)\cap \Ran\phi$, $s_0\in
  S^*(u)$. It follows that $S(u)\subset \intervalco{\frac{s_0}{2},+\infty}$. It
  suffices to take $s^* \le s_0/2$.
\end{proof}

\begin{Rem}
  \label{uniform2}
  To prove Lemma~\ref{convlemma1},  let us remark
  that we could only use
  inequality~\eqref{step2} at $u=u_0$ for $s = s_0$ fixed.
  Indeed, by continuity,  it directly
  implies  that $s_0 \in S(u)$ for $u$ close to $u_0$.  However, to
  obtain Lemma~\ref{convlemma1} for $\tilde S(u)$ (see
  section~\ref{convmpa2}) instead of $S(u)$, the full strength of the
  deformation lemma is needed.
\end{Rem}

\noindent 
From now on, we have to require condition~\eqref{Ac-xi}.
Let us first show it subsumes~\eqref{Ac2}.

\begin{Lem}
  \label{orthog-derivative}%
  Let $(\xi_i)_{i=1}^k$ be the family of vector fields given by
  \eqref{Ac-xi} and assume \eqref{Ac1} holds.  Then
  \begin{equation*}
    \forall u \in \A,\
    \forall d \perp C_u,\quad
    \sum_{i=1}^k  \biginner{u}{\xi_i(u)} \cdot \xi'_i(u)[d]
    = d - \sum_{i=1}^k \biginner{u}{\xi_i'(u)[d]} \, \xi_i(u).
  \end{equation*}
\end{Lem}
\begin{proof}
  For any $u \in \A$, the fact that $u \in C_u \subset V_u$ and that
  $(\xi_i)_{i=1}^k$ is an orthonormal basis of $V_u$ imply $u = \sum
  \inner{u}{\xi_i(u)} \xi_i(u)$.  Differentiating in a direction $d
  \in \H$, yields
  \begin{equation*}
    d = \sum_{i=1}^k \inner{d}{\xi_i(u)} \; \xi_i(u)
    + \sum_{i=1}^k \inner{u}{\xi_i'(u)[d]} \, \xi_i(u)
    + \sum_{i=1}^k \inner{u}{\xi_i(u)} \cdot \xi'_i(u)[d].
  \end{equation*}
  If $d \perp \spanned C_u$, the first term
  vanishes.  This completes the proof.
\end{proof}

\begin{Prop}
  Properties~\eqref{Ac1} and~\eqref{Ac-xi} imply~\eqref{Ac2}.
\end{Prop}
\begin{proof}
  Let $\delta \in \intervaloo{0,1}$ (property ~\eqref{Ac2} will be
  satisfied whatever value is chosen).  Simple geometrical
  considerations show that there exists a $\gamma \in \intervaloo{0,
    \pi/2}$ such that
  \begin{equation*}
    B(d, \delta \norm{d}) \subset
    \intervalcc{1-\delta, 1+\delta} A_\gamma(d).
  \end{equation*}
  Let $u_0\in \Ran\phi$.  As $u_0 \in \operatorname{int}
  C_{u_0}$, there exist $\alpha > 0$ such that
  $\inner{u_0}{\xi_i(u_0)} > \alpha$ for all~$i$.
  Using the continuity of $\xi_i$ and $\xi'_i$ at $u_0$, we can choose
  $r$ sufficiently small and a $M > 0$ (depending only on $u_0$) so
  that, for all $u \in B(u_0,r)$ and all $i= 1,\dotsc, k$,
  \begin{align*}
    \inner{u}{\xi_i(u)} > \alpha, \quad
    \norm{\xi_i(u) - \xi_i(\tilde{u}_0)} \le \epsilon,\quad
    \norm{\xi'_i(u)} \le M,  \quad\text{and}\quad
    \norm{\xi'_i(u) - \xi'_i(\tilde{u}_0)} \le \epsilon,
  \end{align*}
  where $\epsilon >0$ is a constant depending only on $\delta$ and
  $u_0$ (to be chosen later).

  Let $\tilde u_0 \in B(u_0, r)$ and $d \in B(0,r)$ such that $d \perp
  C_{\tilde{u}_0}$.   Let $w \in
  C_{\tilde{u}_0 + d} \cap B(u_0, r)$.  One can write $w = e + \sum
  t_i \xi_i(\tilde{u}_0 + d)$ for some $e \in E$ and $t_i \ge
  0$. Let us start by noticing that $t_i =
  \inner{w}{\xi_i(\tilde{u}_0 + d)}$.  Therefore, recalling that
  $\norm{\xi_i} = 1$, one deduces
  \begin{align}
    \bigabs{t_i - \inner{\tilde{u}_0}{\xi_i(\tilde{u}_0)}}
    &\le \bigabs{\biginner{w - \tilde{u}_0}{\xi_i(\tilde{u}_0 + d)}}
    + \bigabs{ \biginner{\tilde{u}_0}{
        \xi_i(\tilde{u}_0 + d) - \xi_i(\tilde{u}_0)}}  \notag\\
    &\le \norm{w - \tilde{u}_0} + \norm{\tilde{u}_0}
    \norm{\xi_i(\tilde{u}_0 + d) - \xi_i(\tilde{u}_0)}  \notag\\
    &\le 2r + (\norm{u_0} + r) \epsilon .
    \label{eq:t-i}
  \end{align}
  Provided that $\epsilon$ and $r$ are chosen small enough, one can
  assume that $2r + (\norm{u_0} + r) \epsilon \le \alpha/3$.
  In particular, this implies $t_i > 2\alpha/3 > 0$.

  Using the integral form of the mean value theorem, we get
  \begin{equation}
    \label{eq:w}
    w = e + \sum_{i=1}^k t_i \xi_i(\tilde{u}_0 +d)
    = e + \sum_{i=1}^k t_i \xi_i(\tilde{u}_0)
    + \int_0^1 \sum_{i=1}^k t_i \, \xi'_i(\tilde{u}_0+sd)[d] \intd s .
  \end{equation}
  The third term can be rewritten as follows:
  \begin{multline*}
    \sum_{i=1}^k \inner{\tilde{u}_0}{\xi_i(\tilde u_0)} \,
    \xi'_i(\tilde{u}_0)[d]
    + \sum_{i=1}^k
    \bigl(t_i - \inner{\tilde{u}_0}{\xi_i(\tilde u_0)} \bigr)\,
    \xi'_i(\tilde{u}_0)[d] \\
    + \int_0^1 \sum_{i=1}^k t_i \,
    \bigl(\xi'_i(\tilde{u}_0+sd)[d] - \xi'_i(\tilde{u}_0)[d]\bigr) \intd s
    =: d_1 + d_2 + d_3 .
  \end{multline*}
  Using Lemma~\ref{orthog-derivative} on $d_1$, one can write
  equation~\eqref{eq:w} as
  \begin{equation*}
    w = e + \sum_{i=1}^k
    \bigl(t_i - \inner{\tilde{u}_0}{\xi_i'(\tilde{u}_0)[d]} \bigr)
    \, \xi_i(\tilde{u}_0)
    + d
    + d_2 + d_3 .
  \end{equation*}
  Since
  \begin{math}
    \bigabs{\inner{\tilde{u}_0}{\xi_i'(\tilde{u}_0)[d]}}
    \le \norm{\tilde{u}_0} \norm{\xi_i'(\tilde{u}_0)} \norm{d}
    \le (\norm{u_0} + r) M r
  \end{math}, %
  we can assume $r$ was chosen small enough so that this is smaller
  that $\alpha/3$.  Recalling that $t_i > 2\alpha/3$, one sees that
  the coefficients of $\xi_i(\tilde{u}_0)$ are positive and therefore
  \begin{math}
    e + \sum
    \bigl(t_i - \inner{\tilde{u}_0}{\xi_i'(\tilde{u}_0)[d]} \bigr)
    \, \xi_i(\tilde{u}_0)
    \in C_{\tilde u_0}
  \end{math}. %

  The proof is complete if we show that~$d + d_2 + d_3 \in B(d,
  \delta \norm{d})$.  Using~\eqref{eq:t-i}, we deduce
  $\abs{t_i} \le
  \bigabs{t_i - \inner{\tilde{u}_0}{\xi_i(\tilde{u}_0)}}
  + \norm{\tilde u_0} \le 2r + (\norm{u_0} + r) (\epsilon + 1)$.
  Thus the following estimates
  \begin{align*}
    \norm{d_2}
    &\le \sum_{i=1}^k
    \bigabs{t_i - \inner{\tilde{u}_0}{\xi_i(\tilde u_0)}}\,
    \norm{\xi'_i(\tilde{u}_0)} \norm{d}
    \le k \bigl(2r + (\norm{u_0} + r) \epsilon\bigr) M \, \norm{d},
    \\
    \norm{d_3}
    &\le \sum_{i=1}^k \abs{t_i}
    \sup_{s \in \intervalcc{0,1}}
    \norm{\xi'_i(\tilde{u}_0+sd) - \xi'_i(\tilde{u}_0)} \norm{d}
    \le k \bigl(2r + (\norm{u_0} + r) (\epsilon + 1)\bigr)
    \epsilon \norm{d},
  \end{align*}
  show that $\norm{d_i} \le \tfrac{1}{2}\delta \norm{d}$, $i=2,3$,
  provided that the constants $\epsilon$ and $r$ were chosen small
  enough.
\end{proof}

Lemma~\ref{convlemma2-xi} is the second key element
to prove the convergence up to a
subsequence.

\begin{Lem}
  \label{convlemma2-xi}
  Let $\phi$ be a peak selection for $(C_u)_{u \in\A}$ which verifies
  conditions~\eqref{Ac1} and \eqref{Ac-xi}.  Let
  $(u_n)_{n\in\IN}$ and $(s_n)_{n\in\IN}$ be given by the generalized
  MPA (algorithm~\ref{mpa}) with $s_n \in S(u_n)$ for all $n$. Let us assume
  that $\phi$ is continuous, $\overline{\Ran\phi} \subset \A$, and
  either
  \begin{subequations}
    \begin{align}
      &\label{eqsup}
      \exists \tau_1,\dots, \tau_k \in \intervaloo{0,+\infty},\
      \dist \Bigl(\Bigl\{\sum_{i=1}^k \tau_i \xi_i(u) \Bigm|
      u\in\Ran\phi \Bigr\}, \partial \A\Bigr) >0,
      \\
      \text{or}\qquad
      &\label{eq:bdd-E}
      \begin{cases}
        \forall (v_n) \subset \Ran\phi,\
        \bigl(\E(v_n)\bigr) \text{ is bounded from above}
        \limplies
        (v_n) \text{ is bounded}
        \\
        \text{and }  \dim E < \infty.
      \end{cases}
    \end{align}
  \end{subequations}
  If $\sum_{n=0}^{+\infty}s_n< +\infty$ then $(u_n)_{n\in\IN}$ converges in $\A$.
\end{Lem}

\begin{proof}
  Let $k$ be given by the assumption~\eqref{Ac-xi}.
  For $i=1,\dots,k$, set
  $v_{i,n}:=\xi_i(u_n)$,
  and $d_n := - \frac{\nabla \E(u_n)}{\norm{\nabla\E(u_n)}}$.
  Let $r$ be given by assumption~\eqref{Ac-xi}~(v) and $K_i$ be a
  bound for $\xi'_i$.  Denote $K := \max_{i=1,\dotsc, k} K_i$.

  By assumption~\eqref{Ac-xi} and as
  $\phi(u_n + s_n d_n)\in \operatorname{int} C_{u_n + s_n d_n}$, we have
  \begin{equation*}
    v_{i,n+1} = \xi_i\bigl(\phi(u_n+s_nd_n)\bigr)
    = \xi_i(u_n + s_n d_n)
  \end{equation*}
  for any $n\in\IN$.
  Let $n^*$ be large enough so that $s_n < r$.
  Thus, for all $n \ge n^*$,
  \begin{equation}
    \label{eqProjector}
    \norm{v_{i,n+1}-v_{i,n}}
    \le K \norm{s_nd_{n}}
    = K s_n.
  \end{equation}
  Since $\sum s_n < + \infty$, it follows that
  for any $i = 1,\dotsc,k$,
  $(v_{i,n})_{n\in\IN}$ is a Cauchy sequence and therefore converges
  to, say, $v_{i,\infty}$.

  Let us assume~\eqref{eqsup} holds.
  Consider
  $\tilde{v}_n := \sum_{i=1}^k \tau_i v_{i,n}
  = \sum_{i=1}^k \tau_i \xi_i(u_n)$.  It converges and
  its limits belongs to $\A$.
  Since $\phi(\tilde{v}_n)=\phi(u_n)=u_n$ and
  $\phi$ is continuous, the sequence $(u_n)_{n\in\IN}$ converges.
  Its limit lies in $\overline{\Ran\phi}$ and thus in $\A$.

  If on the other hand \eqref{eq:bdd-E} holds, the fact that the
  sequence $(\E(u_n))$ is decreasing implies that $(u_n)$ is bounded.
  Let $(u'_n)$ be a subsequence of $(u_n)$.  Since $u'_n \in
  C_{u'_n}$, one can write $u'_n = e'_n + \sum_{i=1}^k t'_{i,n}
  \xi_i(u'_n)$ for some $e'_n\in E$ and $t'_{i,n} \in \intervaloo{0,+\infty}$.
  As $(u'_n)$ is bounded, so are $(e'_n)$ and $\abs{t'_{i,n}} =
  \abs{\inner{u'_n}{\xi_i(u'_n)}} \le \norm{u'_n}$.
  So, up to subsequences, $(e'_n)_{n\in\IN}$ and
  $(t'_{i,n})_{n\in\IN}$ converge to, say,
  $e'_\infty$ and $t'_{i,\infty}$.
  Thus $u'_n \to u'_\infty := e'_\infty + \sum t'_{i,\infty}
  v_{i,\infty}$.  Thanks to $\overline{\Ran\phi} \subset \A$,
  $u'_\infty \in \A$.  But then the continuity of $\xi_i$ and $\phi$ imply
  \begin{equation}
    \label{eq:lim-sum s_n}
    v_{i,\infty} = \xi_i(u'_\infty)
    \qquad\text{and}\qquad
    u'_n = \phi(u'_n) \to \phi(u'_\infty).
  \end{equation}
  If the same reasoning is performed with another subsequence
  $(u''_n)$, \eqref{eq:lim-sum s_n} implies that $\xi_i(u'_\infty) =
  \xi_i(u''_\infty)$ and therefore, in view
  of definition~\ref{defpeak}~\eqref{phi-uniqueness}, $\phi(u'_\infty) =
  \phi(u''_\infty)$.  As the limit does not depend on the subsequence,
  the whole sequence $(u_n)$ converges in~$\A$.
\end{proof}

\begin{Rem}
  If we wanted to
  seek sign-changing solutions using the cones $C_u := \IR^+
  u^+ \oplus \IR^+ u^-$ (as explained in the Introduction), then we
  would not be able to remove the projection factors in the above
  computation of
  $v_{i,n+1}$. This sheds some light on the difficulty of proving
  the convergence of the Modified Mountain Pass
  Algorithm~\cite{Neuberger97}.
\end{Rem}

\begin{Thm}
  \label{convtheorem}
  Assume $\phi : \A \to \A$ is a continuous peak selection s.t.\
  $\overline{\Ran\phi} \subset \A$ and
  the cones $(C_u)_{u \in \A}$ verify the
  conditions~\eqref{Ac1}, \eqref{Ac-xi} and~\eqref{eqsup} or~\eqref{eq:bdd-E}.
  Suppose moreover that $\E \in \C^1(\H; \IR)$ satisfies the
  Palais-Smale condition in $\Ran \phi$ and that $\inf_{u\in \Ran \phi}
  \E(u)>-\infty$.  Then the sequence
  $(u_n)_{n\in \IN}$ given by the generalized mountain pass
  algorithm~\ref{mpa} possesses a  subsequence
  converging to a critical point of $\E$ in $\Ran\phi$.
  In addition, all limit points  of
  $(u_n)_{n\in\IN}$ are critical points of~$\E$.
\end{Thm}

\begin{proof}
  Let us start by showing that $\bigl(\nabla\E(u_n)\bigr)_{n\in\IN}$ converges
  to zero up to a subsequence. If not, we could assume there exist
  $\delta >0$ and $n_0\in\IN$ such that, for any $n\geq n_0$,
  $\norm{\nabla \E(u_n)}>\delta$. Then, for any $n\geq n_0$,
  the deformation lemma~\ref{deflemma} implies
  \begin{equation*}
    \E(u_{n+1})-\E(u_n) \leq -\alpha s_n\delta.
  \end{equation*}
  Thus, summing up,
  \begin{equation*}
    \lim_{n\to +\infty} \E(u_n)-\E(u_{n_0})
    = \sum_{n=n_0}^{+\infty} \E(u_{n+1})-\E(u_n)
    \le -\delta \alpha \sum_{n=n_0}^{+\infty}s_n.
  \end{equation*}
  As the left-hand side is a real number ($\E$ is bounded from
  below on $\Ran \phi$ and decreasing along $(u_n)_{n\in\IN}$), we have
  $\sum_{n=0}^{+\infty}s_n < +\infty$. So, by Lemma~\ref{convlemma2-xi},
  $u_n\to u^*\in \A$
  and $\norm{\nabla \E(u^*)}\ge \delta$.
  By continuity of $\phi$ at $u^* \in \A$,
  we obtain $\phi(u^*)=u^*$ and, so,
  $u^*\in \Ran \phi$. By Lemma~\ref{convlemma1}, there exists a
  neighborhood $V$ of $u^*$ and $s^*>0$ such that
  $S(u)\subset \intervalco{s^*,+\infty}$ for any $u\in V$.
  Consequently, there exists $n_0$
  such that, for any $n\geq n_0$, $s_n \geq s^*$ whence
  $\sum_{n=0}^{+\infty}s_n$ does not converge, which is a
  contradiction.

  In conclusion, there exists a subsequence $(u_{n_k})_{k\in\IN}$ of
  $(u_n)_{n\in\IN}$ s.t.\ $\norm{\nabla \E(u_{n_k})} \to 0$
  when $k\to +\infty$. As $\E$ satisfies the Palais-Smale
  condition, $(u_{n_k})_{k\in\IN}$ possesses a subsequence converging
  to a critical point of $\E$.

  Concerning the second statement of the theorem, the argument is very
  similar. Let $(u_{n_k})_{k\in\IN}$ be a convergent subsequence and
  assume on the contrary that $u := \lim_{k\to \infty}u_{n_k}$ is not a
  critical point of $\E$. In that case, on one hand, there exists
  $\delta >0$ and $k_1\in\IN$ such that, for any $k\geq k_1$,
  $\norm{\grad \E(u_{n_k})}\geq \delta$.  By Lemma~\ref{deflemma}, we
  have
  \begin{equation*}
    \forall k\geq k_1,\quad
    \E(u_{n_{k+1}})-\E(u_{n_k}) \leq -\alpha\delta s_{n_k}.
  \end{equation*}
  On the other hand, as $u\in \Ran \phi$, we have by Lemma~\ref{convlemma1} that
  \begin{equation*}
    \exists\ s^*>0,\ \exists k_2\in\IN,\ \forall k\geq k_2,\ s_n\geq s^*.
  \end{equation*}
  So, for large $k$, $\E(u_{n_{k+1}})-\E(u_{n_k})\leq
  -\frac{\alpha}{2}\delta s^*$, which is a contradiction because
  $\bigl(\E (u_n)\bigr)_{n\in\IN}$ is a convergent sequence.
\end{proof}

\begin{Rem}
 By previous remarks~\ref{uniform1} and~\ref{uniform2}, we conclude
that we could get the convergence up to a subsequence using the
equation~\eqref{step} only at $u_0$. Thus, the uniform form of
Lemma~\ref{deflemma} is not required (and we could
consider that $\phi(u)$ is a local maximum of $\phi$ on
$C_u$ instead of a global maximum). Nevertheless,  we have kept the
uniform setting along the paper as it will be required in
Section~\ref{convmpa2}.
\end{Rem}

\noindent
The following special case of~\eqref{Ac-xi} is important for the
applications.
\begin{equation}
  \tag{$AC_{4}$}\label{Ac4}
  \left.\begin{minipage}{0.83\linewidth}
      There exist a closed subspace $E\subset \H$ (possibly
      infinite dimensional) and
      linear continuous projectors $P_i :\H\to E^{\perp}$,
      $i=1,\dots,k$, for some $k\in\IN$,  such that
      \begin{itemize}
      \item $\forall\ u\in \H$, $P_i(u)\perp P_j(u)$ whenever $i\ne j$;
      \item $E \oplus \sum_{i=1}^k \Ran P_i = \H$.
      \end{itemize}
      For all $u \in \H$, set $C_u := E
      \oplus \bigl\{ \sum  t_i P_i(u) \bigm| t_i\ge 0 \text{ for all }
      i\bigr\}$.
    \end{minipage}
  \ \right\}
\end{equation}

Let us now sketch the proof that~\eqref{Ac4} implies both~\eqref{Ac1}
and~\eqref{Ac-xi}.
Consider
$\A := \{ u \in \H \mid P_i(u) \ne 0 \text{ for all } i \}$ and
$\xi_i(u) := \frac{P_i(u)}{\norm{P_i(u)}}$.
Clearly $\xi_1, \dotsc, \xi_k$ are $\C^1$ functions on $\A$.
Moreover $e + \sum t_i P_i(u) \in \operatorname{int} C_u$ if and only if
all $t_i > 0$.  Since $u = P_E(u) + \sum_{i=1}^k P_i(u)$ where $P_E$
denotes the orthogonal projection on $E$, one has $u \in
\operatorname{int} C_u$.
Given the definitions of $\A$ and $\xi_i$, points (i), (iii) and~(iv)
of~\eqref{Ac-xi} are straightforward.
A simple computation shows that
$\xi'_i(u)\bigl[\sum t_j P_j(u)\bigr]$ is a multiple of $P_i(u)$
whence~(ii) follows.
Finally, as $\norm{\xi'_i(u)} = O(1/\norm{P_i(u)})$, (v) will hold
provided $\norm{P_i(u)}$ is bounded away from $0$ when $u \in \Ran\phi$.
Note that this latter condition also ensures that
$\overline{\Ran\phi} \subset \A$.
Remark that these cones satisfy property \eqref{eqsup} with $\tau_1 =
\cdots = \tau_k = 1$ because $\dist(\sum \xi_i(u), \partial\A)
= \min_j \norm{P_j(\sum \xi_i(u))} = 1$.

Thus, as a corollary of Theorem~\ref{convtheorem}, we get the
following proposition.
It can be thought as an abstract version of the convergence
results in~\cite{sym1,zhou1,zhou2}.

\begin{Prop}
  \label{conv-P_i}
  Let us consider $\phi : \A \to \A$ a continuous
  peak selection with the cones $(C_u)_{u \in \A}$ being given by
  condition~\eqref{Ac4} and $\A := \{u\in \H \mid
  P_i(u)\ne 0 \text{ for all } i=1,\dots,k\}$.
  Assume that
  $\inf_{u\in \Ran \phi} \norm{P_i (u)}>0$ for all
  $i=1,\dots,k$, that $\E \in \C^1(\H; \IR)$ satisfies the
  Palais-Smale condition in $\Ran \phi$ and that $\inf_{u\in \Ran \phi}
  \E(u)>-\infty$.   Then the sequence
  $(u_n)_{n\in \IN}$ given by the generalized mountain pass
  algorithm~\ref{mpa} possesses a  subsequence
  converging to a critical point of $\E$ in $\Ran\phi$.
  In addition, all limit points  of
  $(u_n)_{n\in\IN}$ are critical points of~$\E$.
\end{Prop}

In Theorem~\ref{convtheorem}, the Palais-Smale condition is required. For
the particular  case of  $\H = H^1(\IR^N)$,  this
condition does not generally hold as mass may be lost at
infinity. Fortunately,
$H^1(\IR^N)$ respects the following compactness condition (see
for example the
 paper~\cite{lieb} for a proof): for any
bounded sequence $(u_n)_{n\in\IN}\subset H^{1}(\IR^N)$
staying away from zero, there exists
$(x_n)_{n\in\IN}\subset \IZ^N$ such that $\bigl(u_n(\cdot +
x_n)\bigr)$ weakly converges
up to a subsequence to a non-zero function.  This is
enough to get the convergence up to a subsequence.

\begin{Prop}
  \label{convinrn}
  Assume the hypotheses of Theorem~\ref{convtheorem} hold, except for
  the Palais-Smale condition.
  Let $\H := H^1(\IR^N)$ and $(u_n)_{n\in \IN}$ be the sequence given
  by the Mountain Pass Algorithm~\ref{mpa}.
  If, for any $u\in \H$ and
  $x\in \IZ^N$, $\E\bigl(u(\cdot +x)\bigr) = \E(u)$
  and if $ \H \to \H : u \mapsto \nabla\E(u)$ is continuous for the
  weak topology on $\H$,
  then there exists a sequence $(x_n)_{n\in\IN}\subset
  \IZ^N$ such that $\bigl(u_n(\cdot + x_n)\bigr)_{n\in\IN}$ weakly converges up
  to a subsequence to a nontrivial critical point of $\E$.
 \end{Prop}

\begin{proof}
  We will only briefly sketch the proof.
  As $(u_n)_{n\in\IN}$ is bounded in $H^1(\IR^N)$ and stays away from
  $0$, the compactness condition recalled above implies
  that there exists a sequence $(x_n)_{n\in\IN}\subset \IZ^N$ such
  that $u(\cdot + x_n)$ weakly converges, up to a subsequence, to
  $u^*\neq 0$. Intuitively, the translations ``bring back'' some mass
  that $u_n$ may loose at infinity.

  Using the translation invariance of
  $\E$, the corresponding equivariance of $\nabla\E$ and
  the weak continuity of $\nabla\E$,
  we conclude that $u^*$ is a critical
  point of $\E$.
\end{proof}

\subsection{Convergence of the whole sequence}
\label{convmpa2}

In this section, we refine the stepsize used previously to get the
convergence of the whole sequence generated by
algorithm~\ref{mpa}. We require that the stepsize $s_n \in
\tilde{S}(u_0) := \tilde{S}^{*}(u_0)\cap \left(\frac{1}{2}\sup
  \tilde{S}^{*}(u_0), +\infty \right)$ where
\begin{equation*}
  \begin{split}
    \tilde{S}^{*}(u_0)
    := \Bigl\{ s_0>0  \Bigm|  \forall\ s \in \intervaloc{0, s_0},\
    & u_s := u_0 -
    s\frac{\nabla\E(u_0)}{\norm{\nabla\E(u_0)}} \in \A
    \text{ and }\\
    &\E\bigl(\phi(u_s)\bigr) - \E(u_0)
    < -\alpha s\norm{\nabla \E(u_0)}   \Bigr\}  .
  \end{split}
\end{equation*}
Using the
deformation lemma~\ref{deflemma}, we get that $\tilde S(u_n) \ne
\emptyset$ as long as $u_n$ is not a critical point.
Moreover, working as previously, we get
results~\ref{convlemma1}, \ref{convlemma2-xi} and \ref{convtheorem} for
this new choice of stepsizes. Let
us remark that, this time,  we really need that
inequality~\eqref{step} is valid in a neighborhood of $u_0$ to get
Lemma~\ref{convlemma1}.
This new stepsize will allow us  to
control the energy for any $0<s\leq s_0$.
Under a ``localization'' assumption, we now prove that
the whole sequence $(u_n)_{n\in\IN}$ given by the
mountain pass algorithm~\ref{mpa}
converges to a nontrivial critical point of $\E$.

\begin{Thm}
  \label{convthm}
  Assume that $u$ is the unique critical point of $\E$ in the ball $B(u,
  \delta)$ for some $\delta >0$.
  Under the same assumptions as those of Theorem~\ref{convtheorem}, if there
  exists $n^*\in \IN$ such that $\E(u_{n^*})<a:= \inf_{v\in
    \partial B(u,\delta)\cap \Ran \phi}\E(v)$ and $u_{n^*}\in B(u,\delta)$
  then the  sequence $(u_n)_{n\in\IN}$ produced by
  algorithm~\ref{mpa} with stepsizes $s_n \in \tilde{S}(u_n)$
  converges to $u$.
\end{Thm}

\begin{proof}
  For any $m> n^*$, we claim that $u_m \in B(u, \delta)$. If not,
  as $u_{n^*}\in B(u,\delta)$,
  there exists $m\geq n^*$ such that $u_m \in B(u, \delta)$ and
  $u_{m+1}= \phi\bigl(u_m -s_m\frac{\grad \E(u_m)}{
    \norm{\grad \E(u_m)}} \bigr) \notin
  B(u, \delta)$, with $s_m\in \tilde{S}(u_m)$.
  By continuity, there exists $0<s\leq s_m$ such that
  $\phi\bigl(u_m -s\frac{\grad \E(u_m)}
  {\norm{\grad \E(u_m)}} \bigr)\in \partial B(u, \delta)\cap \Ran \phi$.
  This is a contradiction
  because, by the definition of $s_m$ and as $\E$ is
  decreasing along $(u_n)_{n\in\IN}$,  we have
  $a\leq \E\bigl(\phi(u_m -s\frac{\grad \E(u_m)}{\norm{\grad
      \E(u_m)}} )\bigr)\leq \E(u_m)\leq  \E(u_{n^*})<a$.

  As $u$ is the unique critical point in $B(u,\delta)$, by
  Theorem~\ref{convtheorem}, $u$ is the unique accumulation point of
  $(u_n)_{n\in\IN}$. So, $u_n$ converges to $u$.
\end{proof}

\section{Applications}
\label{applmpa}

\subsection{Application to Indefinite Problems}

For problem~\eqref{eqSzulkin},
the energy functional $\E$ given by~\eqref{functional-indefinite}
is defined on $\H:=
H^1_0(\Omega)$. Let us denote the decomposition $\H= \H^{(-)} \oplus \H^{(+)}$
corresponding to the spectral decomposition of
$-\Delta +V$ with respect to the positive and negative part of the
spectrum.  For any $u$, we let $u^{(-)}\in \H^{(-)}$ and $u^{(+)}\in
\H^{(+)}$ be the unique elements
such that $u = u^{(-)}+u^{(+)}$.
Let us remark that the case $\H^{(-)} = \{0\}$ corresponds the
traditional mountain pass algorithm with a positive definite linear operator.

We choose the following peak selection.
Let $\A := \H \setminus \H^-$
and, for any $u\in \A$, let $C_u$ be the cone
$C_u := \H^{(-)} \oplus \IR^+ u = \H^{(-)} \oplus \IR^+ u^{(+)}$.
The peak selection $\phi$ for $(C_u)$ is the map
\begin{equation*}
  \phi: \A\to \A : u\mapsto \phi(u)
\end{equation*}
such that, for all $u \in \A$,
$\phi(u)$ maximizes $\E$ on $C_u$.   To prove that $\phi$ is
continuous, we refer to the original paper~\cite{szulkin}.
Is is easy to check that these cones verify
\eqref{Ac4}.  Indeed it suffices to consider $E=
\H^{(-)}$, $k=1$ and $P_1 : \H\to E^\perp$, the orthogonal
projection on~$E^\perp$.

To apply Proposition~\ref{conv-P_i},  we need to verify the
following assumptions  on $\E$: on a bounded domain $\Omega$,
\begin{enumerate}
\item it is standard to show that $\E\in\C^1$;
\item $\E$ verifies the Palais-Smale condition on $\Ran \phi$
  (see~\cite{szulkin});
\item $\inf_{u\in \Ran \phi}\E(u)>-\infty$: actually $\E$ is bounded from
  below by $0$ on $\Ran \phi$,
  see~\cite{szulkin};
\item $0$ does not belong to $\overline{\Ran P_1 \circ \phi}$:
  it comes from the
  fact that $0$ is a strict local minimum of $\E$ on $E^{\perp}=\H^{(+)}$
  (see~\cite{szulkin}).
\end{enumerate}
In conclusion, Proposition~\ref{conv-P_i} applies and gives
the convergence up to a
subsequence of the sequence $(u_n)$ generated by generalized
mountain pass algorithm~\ref{mpa} for this indefinite
problem provided that the domain $\Omega$ is bounded.

Let us now sketch what happens about the convergence up to a
subsequence when $\Omega = \IR^N$. As $(\E(u_n))_{n\in\IN}$ is
decreasing (see~\ref{stability}) and is bounded away from zero, we
have that $(u_n)_{n\in\IN}$ is  bounded and stays away from zero in
$H^1(\IR^N)$ (see~\cite{szulkin}).  On the other hand, $V$ is assumed
to be $1$-periodic, thus $\E\bigl(u(\cdot + x)\bigr) = \E(u)$
for any $u \in \H$ and $x \in \IZ^N$.  It is not difficult to check
that $\nabla\E$ is weakly continuous.  Thus, Theorem~\ref{convinrn}
asserts that, if $(u_n)$ is the sequence generated by the MPA,
there exists a sequence of translations $(x_n) \subset \IZ^N$ such that
$(u_n(\cdot + x_n))_{n\in\IN}$
weakly converges, up to a subsequence, to a nontrivial critical point
$u^*$ of $\E$.
Moreover, if $\E(u_n)
\to \inf_{u\in\Ran \phi}\E(u)$, then it can be proved
that the above convergence is strong.  The idea is that,
if it does not converge strongly, some mass is lost at
infinity.  At the limit, this mass will take away a quantity of
energy greater or equal to
$\inf_{u \in \Ran\phi} \E(u) > 0$, a contradiction.

\medskip

\paragraph{Numerical experiments}
Let us start by giving some details on the computation of various
objects intervening in the MPA.
Functions in $\H$ will be approximated using
$P^1$-finite elements on a Delaunay triangulation of $\Omega$
generated by Triangle~\cite{Triangle}.
The matrix of the quadratic
form $(u_1, u_2) \mapsto \int_{\Omega} \nabla u_1 \nabla u_2$ is
readily evaluated on the finite elements basis.
For $(u_1, u_2) \mapsto \int_{\Omega} V(x) u_1 u_2 \intd x$
and the various integrals involving $u$ to a power,  a quadratic
integration formula on each triangle is used.
The gradient $g := \nabla\E(v)$ is computed in the usual way:
the function $g \in \H$ is the solution of the linear system of equations
$\forall \varphi \in \H$, $(g | \varphi)_{\H} = \mathrm{d}
\E(v)[\varphi]$.
In practice,  the peak selection  $\phi$ must be evaluated with great
accuracy to obtain satisfying
results.  For this, we use a
limited-memory quasi-Newton code for bound-constrained
optimization~\cite{L-BFGS-B}.
The program stops when the gradient
of the energy functional at the approximation has a norm less than
$10^{-4}$.

As an illustration, we consider $\Omega = \intervaloo{0,1}^2$,
$V \in\IR$ constant and $p=4$.  Let us remark that $\H^{(-)}$
is then formed by eigenfunctions of $-\Delta +V$ with negative eigenvalues.
In dimension~$2$, the eigenvalues $\lambda_i$ of $-\Delta$ on the
square $\intervaloo{0,1}^2$ with zero Dirichlet boundary conditions
are given
by $\pi^2 (n^2 + m^2)$ with $n,m = 1,2,\dots$\@ The
related eigenfunctions are given by $\sin (n\pi x )\sin(m\pi y)$.
We get $\lambda_1 = 2\pi^2 \approx 19.76$,
$\lambda_2 = \lambda_3 = 5\pi^2 \approx 49.48$ (a double
eigenvalue),
$\lambda_4 = 8\pi^2 \approx 78.95$,
$\lambda_5 = \lambda_6 = 10\pi^2 \approx 98.69$,...

Figure~\ref{carrelam} depicts four
non-zero solutions  approximated by the algorithm~\ref{mpa} for four
different values of $V$.  The algorithm was always started from
$u_0(x,y) := xy(x-1)(y-1)$.
The graphs on the left-hand side are
given for the values $V = 0$ ($\dim \H^{(-)} = 0$) and
$-\lambda_2 < V = -21 < -\lambda_1$ ($\dim \H^{(-)}=1$). The
graphs on the right-hand side are given for $-\lambda_4 < V = -50
< -\lambda_3$ ($\dim\H^{(-)} = 3$) and $-\lambda_5 < V = -80
< -\lambda_4 $ ($\dim \H^{(-)} =4$). In Table~\ref{tablelam}, we
present some characteristics of the solutions.

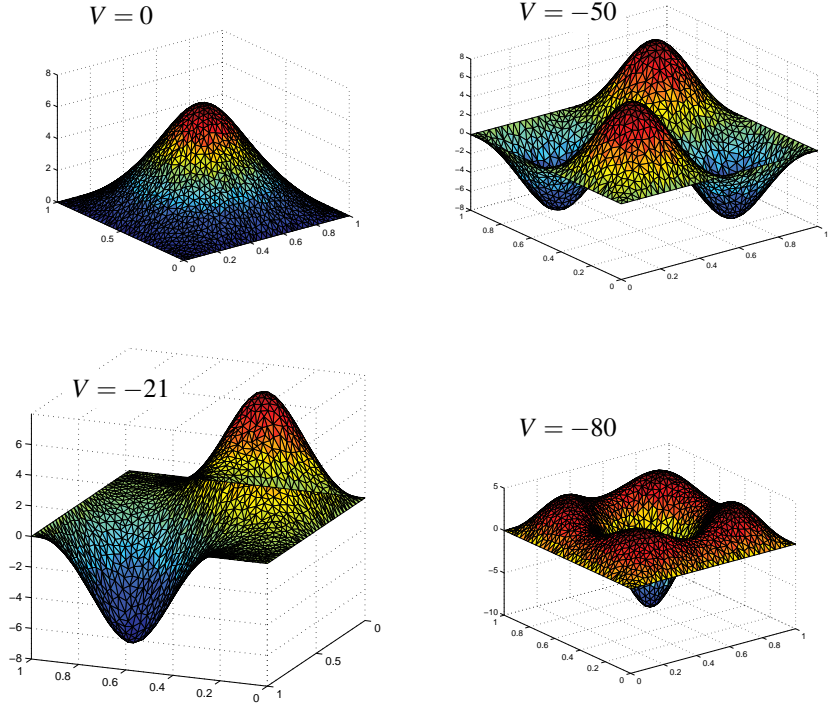
\begin{figure}[h]
  \null\hfill
  \begin{tikzpicture}
    \node at (0,0) {\includegraphics[width=0.4\linewidth]{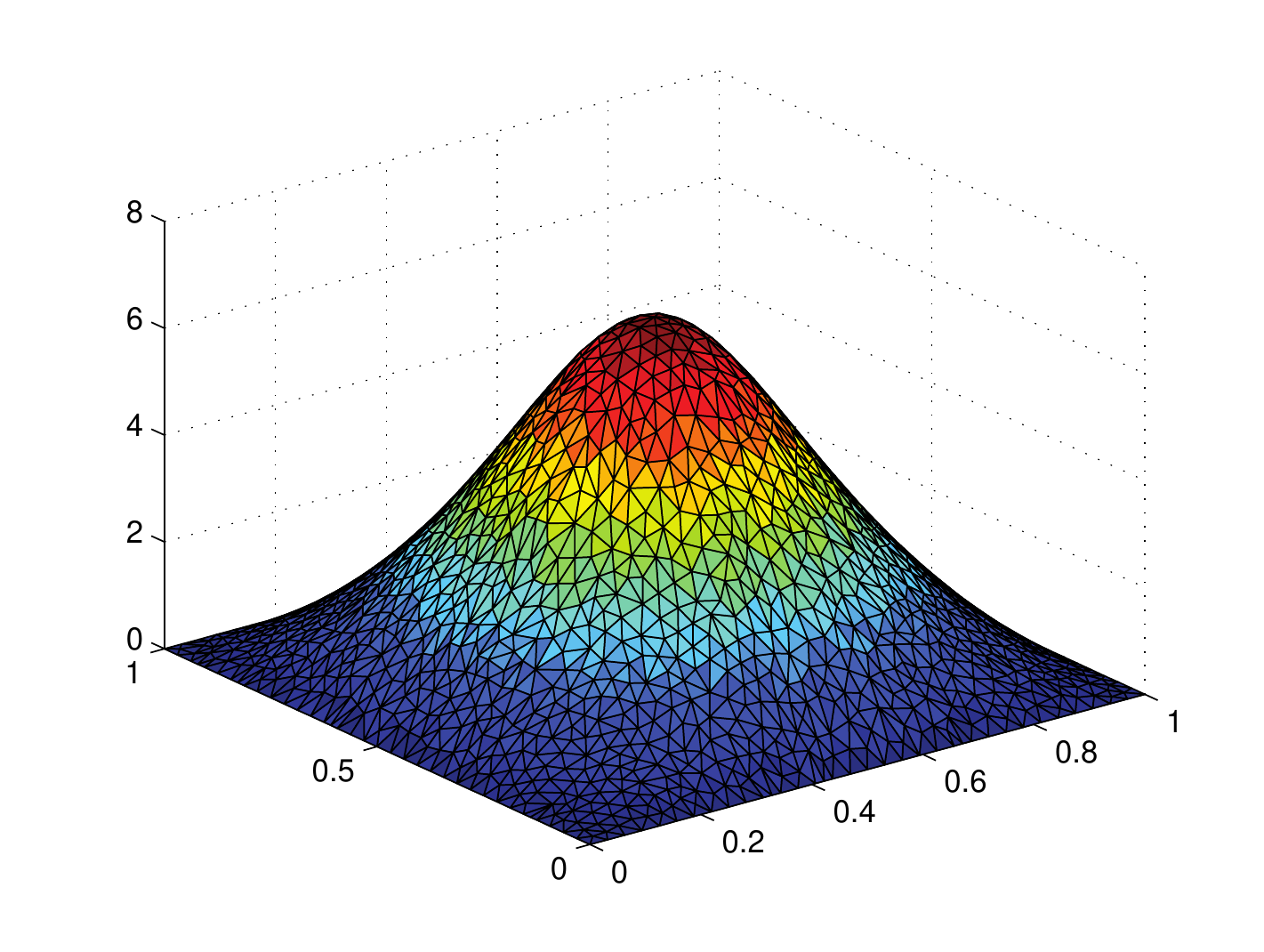}};
    \node[fill=white] at (-10mm,18mm) {$V=0$};
  \end{tikzpicture}
  \hfill
  \begin{tikzpicture}
    \node at (0,0) {\includegraphics[width=0.4\linewidth]{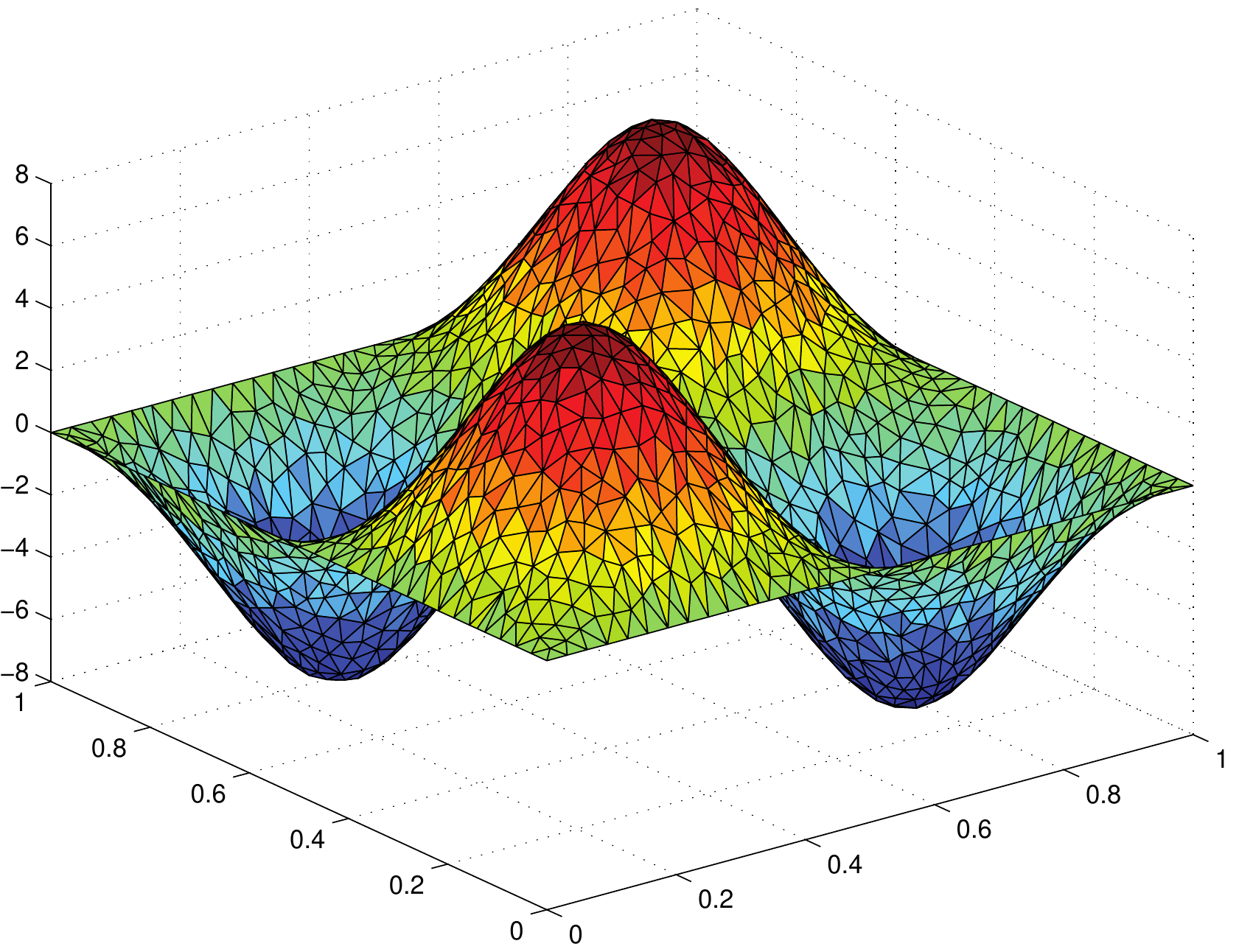}};
    \node[fill=white] at (-10mm,18mm) {$V=-50$};
  \end{tikzpicture}
  \hfill\null\\
  \null\hfill
  \begin{tikzpicture}
    \node at (0,0) {\includegraphics[width=0.4\linewidth]{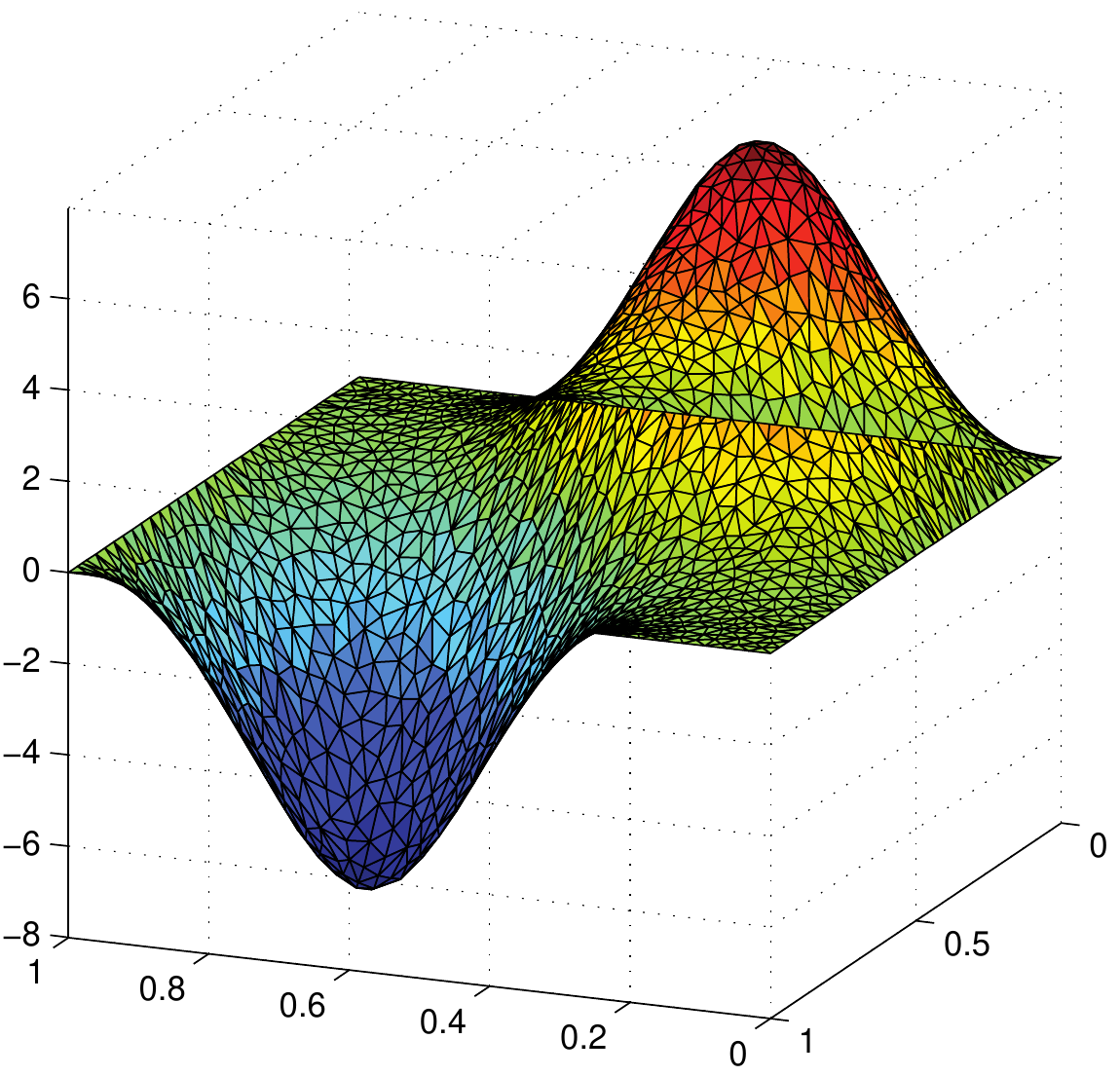}};
    \node[fill=white] at (-10mm,18mm) {$V=-21$};
  \end{tikzpicture}
  \hfill
  \begin{tikzpicture}
    \node at (0,0) {\includegraphics[width=0.4\linewidth]{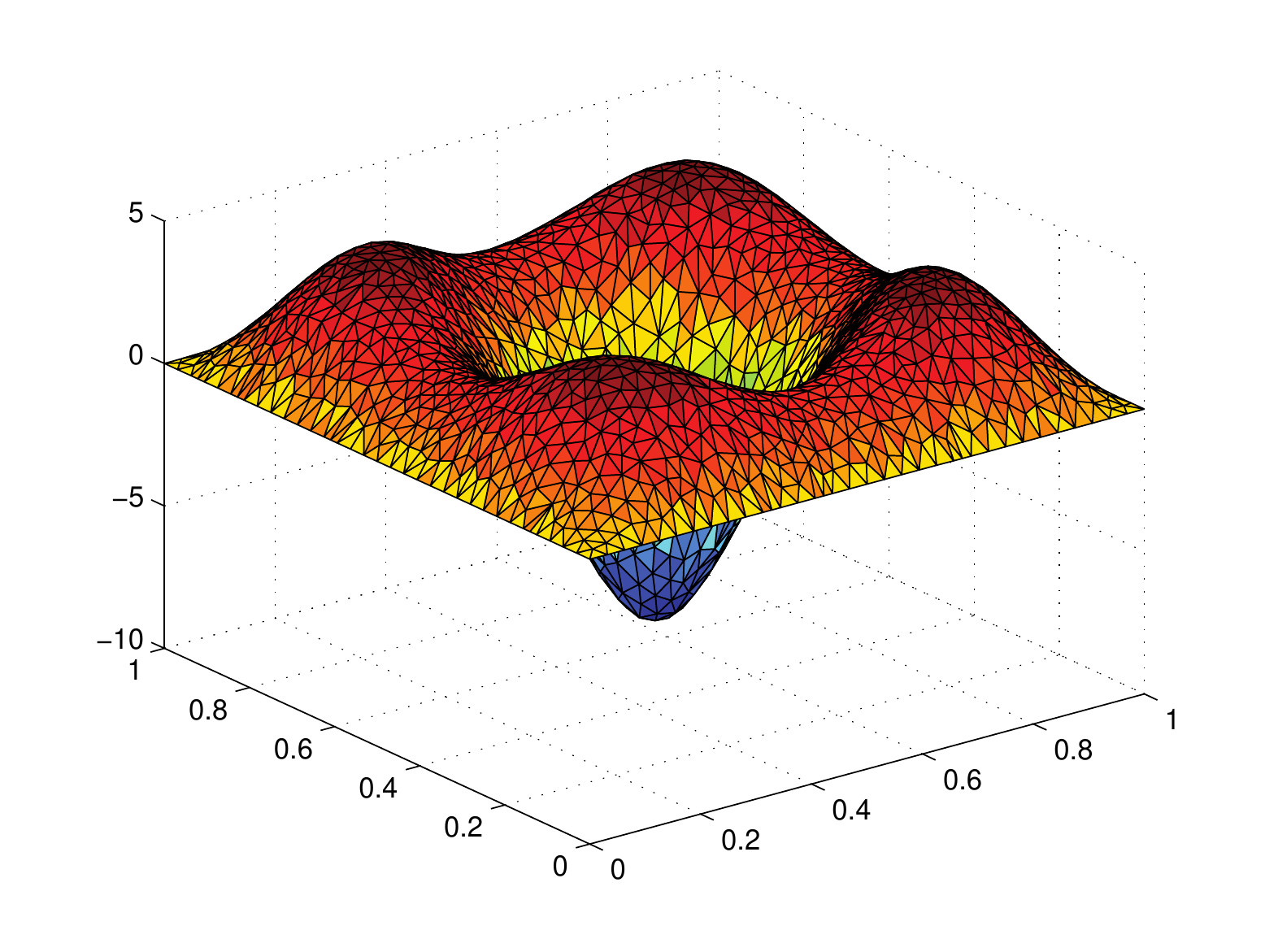}};
    \node[fill=white] at (-10mm,18mm) {$V=-80$};
  \end{tikzpicture}
  \hfill\null
  \caption{MPA solutions for an indefinite problem on a square}
  \label{carrelam}
\end{figure}

\begin{table}[h]
  \begin{center}
    \begin{tabular}{>{\vrule height 2.3ex depth 0.9ex width 0pt}
        r r@{.}>{$}l<{$} r r@{.}l}
      \multicolumn{1}{c}{$V$}&
      \multicolumn{2}{c}{$\norm{\nabla \E}$}
      & \multicolumn{1}{c}{\# of steps} & \multicolumn{2}{c}{$\E(u)$}\\
      \hline $0$ &   6&0\cdot 10^{-5} & 7& 37&89\\
      \hline $-21$ & 6&4\cdot 10^{-5} & 48& 70&43\\
      \hline $-50$ & 5&3\cdot 10^{-5} & 113& 91&42\\
      \hline $-80$ & 6&5\cdot 10^{-5} & 44& 35&06
    \end{tabular}
  \end{center}
  \caption{Characteristics of approximate solutions to an indefinite problem.}
  \label{tablelam}
\end{table}

For $V = 0$, we remark that the approximation is even w.r.t.\
any symmetry of the square
and is positive. It was expected and it is actually already known in this case
(i.e.\  for the problem $-\Delta u = \abs{u}^{p-2}u$) that
ground state solutions have the same symmetries
as the first eigenfunctions of $-\Delta$ (see~\cite{gnn,bbgv}).

For $V = -21$, the approximation has two nodal domains and a
diagonal as nodal line. It seems to respect the
symmetries of a second eigenfunction of~$-\Delta$. It can be explained
as follows.  When $V = 0$, it is
proved~\cite{bbgv} that, for $p$ close to $2$, least energy nodal solutions
have the same symmetries as
their projections on the second eigenspace of~$-\Delta$. On the
square, it is even conjectured that the projection  must be a
function odd w.r.t. a diagonal.
In view of the bifurcation diagrams computed by
J.~M.~Neuberger~\cite{Neuberger-GNGA, Neuberger-Newton},
the least energy nodal solution for $V \in \intervaloc{-\lambda_1, 0}$
becomes the solution with lowest energy when
$V \in \intervaloc{-\lambda_2, -\lambda_1}$ and no bifurcation happens
along the way.  So it is reasonable (and this is supported by
the bifurcation diagrams) that they keep the same symmetries
along the whole branch.

We also observe that, for $V = -50$  (resp.\ $-80$), the
approximation  seems to respect the
symmetries of (and has the ``same form'' as) a fourth (resp.\ fifth)
eigenfunction of $-\Delta$. Their
number of bumps corresponds to their Morse index
($\dim\H^{(-)} + 1$).

All those considerations support the conjecture that
if $-\lambda_{n}< V < -\lambda_{n-1}$ then, at least for $p$
  small enough, ground state solutions respect the symmetries of a
  $n^{\text{th}}$ eigenfunction of $-\Delta$.


\subsection{Application to Systems}

In this section we will perform numerical experiments for the
system~\eqref{eqNoris-general}.
The corresponding energy
functional~\eqref{functional-system} is defined on
$\H= H^1_0(\Omega,\IR^k)$ endowed with the norm $\norm{u}^2 =
\int_{\Omega} \abs{\nabla u}^2 = \sum_i \int_{\Omega} \abs{\nabla
  u_i}^2\intd x$.  In \cite{noris}, B.~Noris and G.~Verzini
prove that the minimization of $\E$ on $\N := \bigl\{u\in \A \bigm|
\forall i=1,\dots,k,\ \int_{\Omega}
\abs{\nabla u_i}^2\intd x = \int_{\Omega} \partial_i F(u)u_i\intd x
\bigr\}$, where
$\A  := \{u\in \H \mid u_i\neq 0 \text{ for every } i\}$,
yields a solution $u=(u_1,\dots, u_k) = \sum u_i e_i$
with $u_i \neq 0$ for all $i=1,\dots,k$ provided that the following
assumptions are satisfied:
there exist  $p\in \intervaloo{2,2^*}$, $C_F >0$ and
$\delta >0$ such that, for any $u,\lambda \in \IR^k$, one has
\begin{enumerate}
\item\label{F1} $\sum_{i,j}\abs{\partial^2_{i,j} F(u)}\le C_F \abs{u}^{p-2}$,
  $\sum_{i} \abs{\partial_i F(u)}\le C_F \abs{u}^{p-1}$ and
  $\abs{F(u)}\le C_F \abs{u}^p$;
\item $\sum_{i,j} \partial^2_{i,j} F(u)\lambda_i u_i \lambda_j u_j -
  (1+\delta) \sum_{i}\partial_i F(u)\lambda_i^2 u_i\ge 0$;
\item for every $i$ there exists $\bar{u}_i >0$ such that $\partial_i
  F(\bar{u}_i e_i)>0$;
\item\label{F4}
  $\partial_i F(u)u_i \le \partial_i F(u_ie_i)u_i$ for every $i$.
\end{enumerate}
The first three assumptions are traditional in the framework of
variational methods. The last one insures $u_i\ne 0$ for all $i$.
In this section, we will use the Mountain Pass
Algorithm~\ref{mpa} with the following peak selection.
For any $u= (u_1,\dots, u_k)\in \A$, we consider the
cone $C_u := \{(t_1 u_1,\dots, t_k u_k) \mid t_i\ge 0 \text{ for all }
i=1,\dots,k\}$.
The peak selection $\phi$ for $(C_u)_{u\in\A}$ is the map
\begin{equation*}
  \phi: \A\to \A : u\mapsto \phi(u)
\end{equation*}
such that $\phi(u)$ maximizes $\E$ on $C_u$. Under the additional
hypothesis that $\sum_i \partial_i F(u) u_i \ge 0$, the second
assumption plays the role of the traditional super-quadraticity and
implies that $\phi$ is well-defined as a peak selection. In
fact, if $u\in \A$
verifies $\operatorname{d} \E(u)[(\lambda_1
u_1,\dots, \lambda_k u_k)] =0$ for any $(\lambda_1,\dots,\lambda_k)\in \IR^k$
then   $u$ is a strict local maximum of
$\E$ on $C_u$. It implies the uniqueness of the global maximum of $\E$
on $C_u$. Moreover, $\phi$ is continuous.

To see that assumption~\eqref{Ac4} is satisfied, it suffices to
take $E = \{0\}$ and, for
$i=1,\dots,k$, $P_i(u) = P_i\bigl((u_1,\dots, u_k)\bigr) := u_i e_i$
i.e., $P_i$ is the projection on the $i^{\text{th}}$ component of $u$.
Finally,
let us quickly run through the assumptions of Proposition~\ref{conv-P_i}:
\begin{enumerate}
\item it is standard to show that $\E\in\C^1$;
\item $\E$ verifies the Palais-Smale condition on $\Ran \phi$
  (see~\cite{noris});
\item $\inf_{u\in \Ran \phi}\E(u)>-\infty$: actually $\E$ is bounded from
  below by $0$ on $\Ran \phi$ (see~\cite{noris});
\item $\dist (\Ran\phi,\partial\A)>0$
  (see~\cite{noris});
\end{enumerate}
In conclusion, Proposition~\ref{conv-P_i} applies and gives
the convergence, up to a
subsequence, of the sequence $(u_n)$ generated by the
Mountain Pass Algorithm~\ref{mpa}.

\medskip
\paragraph{Numerical experiments}

For the numerical experiments, we will consider the following
particular case of equation~\eqref{eqNoris-general}:
\begin{equation}
  \label{eqNoris}
  \left\{
    \begin{aligned}
      -\Delta u_i(x)
      &= \mu_i u_i^3 + u_i\sum_{j\ne i} \beta_{i,j}u_j^2 ,&&\ x\in\Omega, \\
      u_i(x)&=0,&&\ x \in\partial\Omega ,
    \end{aligned}
  \right.
  \qquad  i=1,\dots,k,
\end{equation}
where $\beta_{i,j} = \beta_{j,i}$ and $\Omega$ is a bounded
domain of~$\IR^2$.
This system is modeling a competition between $k$
populations.  We will focus on the case $\Omega = \intervaloo{0,1}^2$
and $k=2$.  In this setting, the assumptions
(\ref{F1})--(\ref{F4}) stated above boild down to
\begin{equation}
  \label{eq:sys-beta-cond}
  \mu_1 > 0,\qquad
  \mu_2 > 0,\qquad\text{and}\qquad
  -\sqrt{\mu_1 \mu_2} \le \beta_{1,2} \le 0.
\end{equation}
Let us remark that  the condition $\sum_i \partial_i F(u) u_i \ge 0$
discussed in the previous section is also verified in this range.

Let us now give the outcome of the algorithm for various choices of
$(\mu_1, \mu_2, \beta_{1,2})$.  The MPA will always start with the
function $u_0 = (u_{0,1}, u_{0,2}) \in \A$ where
$u_{0,1}(x,y) =  u_{0,2}(x,y) = xy(1-x)(1-y)$ and stops when the
norm of the gradient is less than~$10^{-4}$.

First we choose $(\mu_1, \mu_2, \beta_{1,2}) = (1, 4, -1)$.  The
numerical solution $(u_1, u_2)$ is depicted on Figure~\ref{carre
  sys -1} and some characteristics are given in
Table~\ref{table-sys}.  In this case, the assumptions
\eqref{eq:sys-beta-cond} are satisfied so the fact that the algorithm
converges to a solution $(u_1, u_2)$ with $u_1 > 0$ and $u_2 > 0$ is
expected.
Notice also that the solutions $u_1$ and $u_2$ are even w.r.t.\
axes of symmetry of the square.

As second choice, we consider $(\mu_1, \mu_2, \beta_{1,2}) = (1, 4,
0.5)$.  The MPA solution $(u_1, u_2)$ is depicted on Figure~\ref{carre
  sys 0.5} and some characteristics are given in the second row of
Table~\ref{table-sys}.  Despite the fact that the assumptions
\eqref{eq:sys-beta-cond} are not satisfied anymore, the solution is
similar to the found in the first case.  If we enlarge $\beta_{1,2}$
further and choose $(\mu_1, \mu_2, \beta_{1,2}) = (1, 4, 1.2)$, the
algorithm still converges (see the third row of Table~\ref{table-sys})
but this time, the second component vanishes (see Figure~\ref{carre
  sys 1.2}).  What happens is that, at the very first step, $u_2 = 0$
and then the MPA essentially proceeds as if the system was only
consisting in the first equation.

\begin{figure}[h!t]
  \vspace{-0.13\linewidth}%
  \begin{center}
    \includegraphics[width=0.4\linewidth]{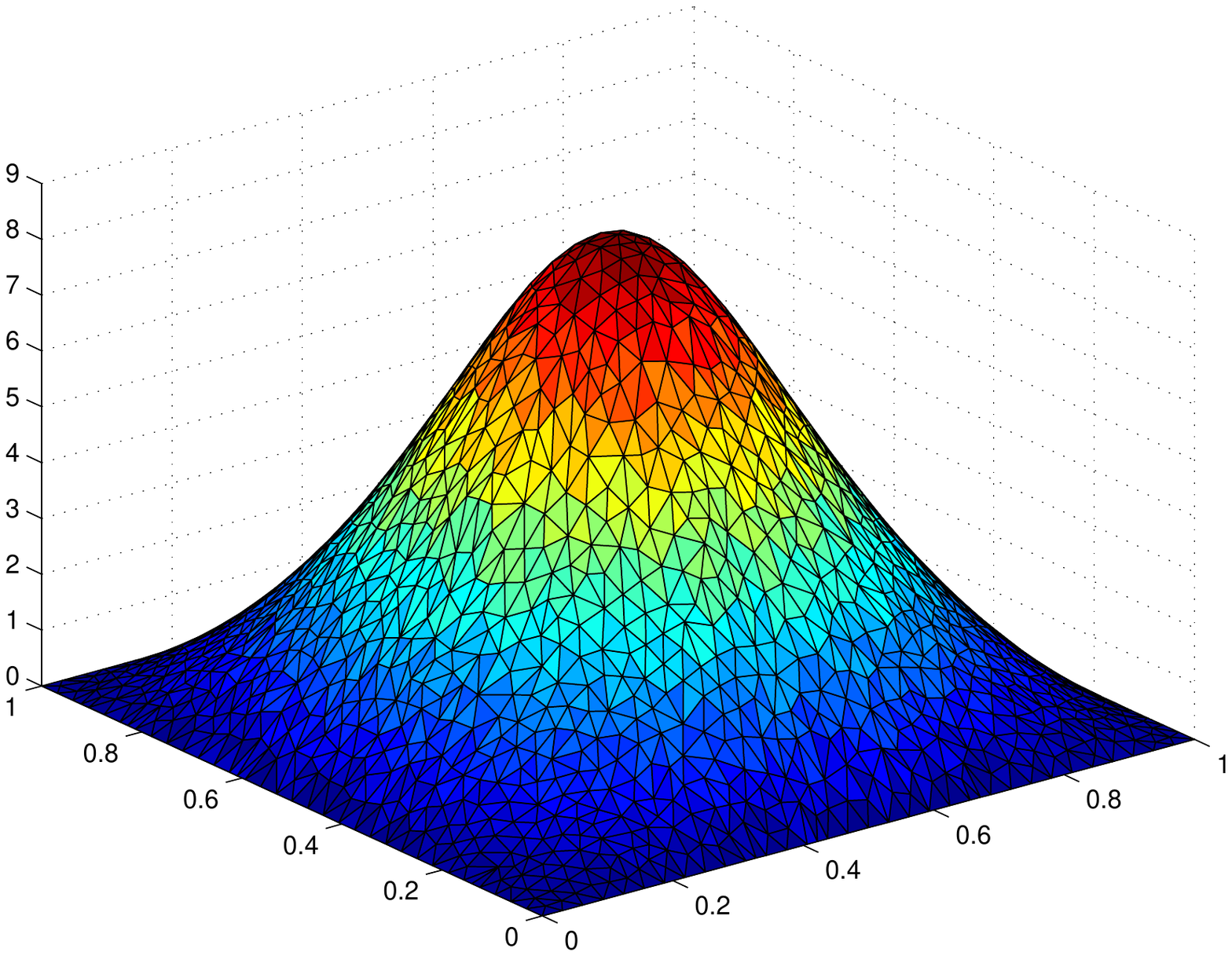}%
    \hfil
    \includegraphics[width=0.4\linewidth]{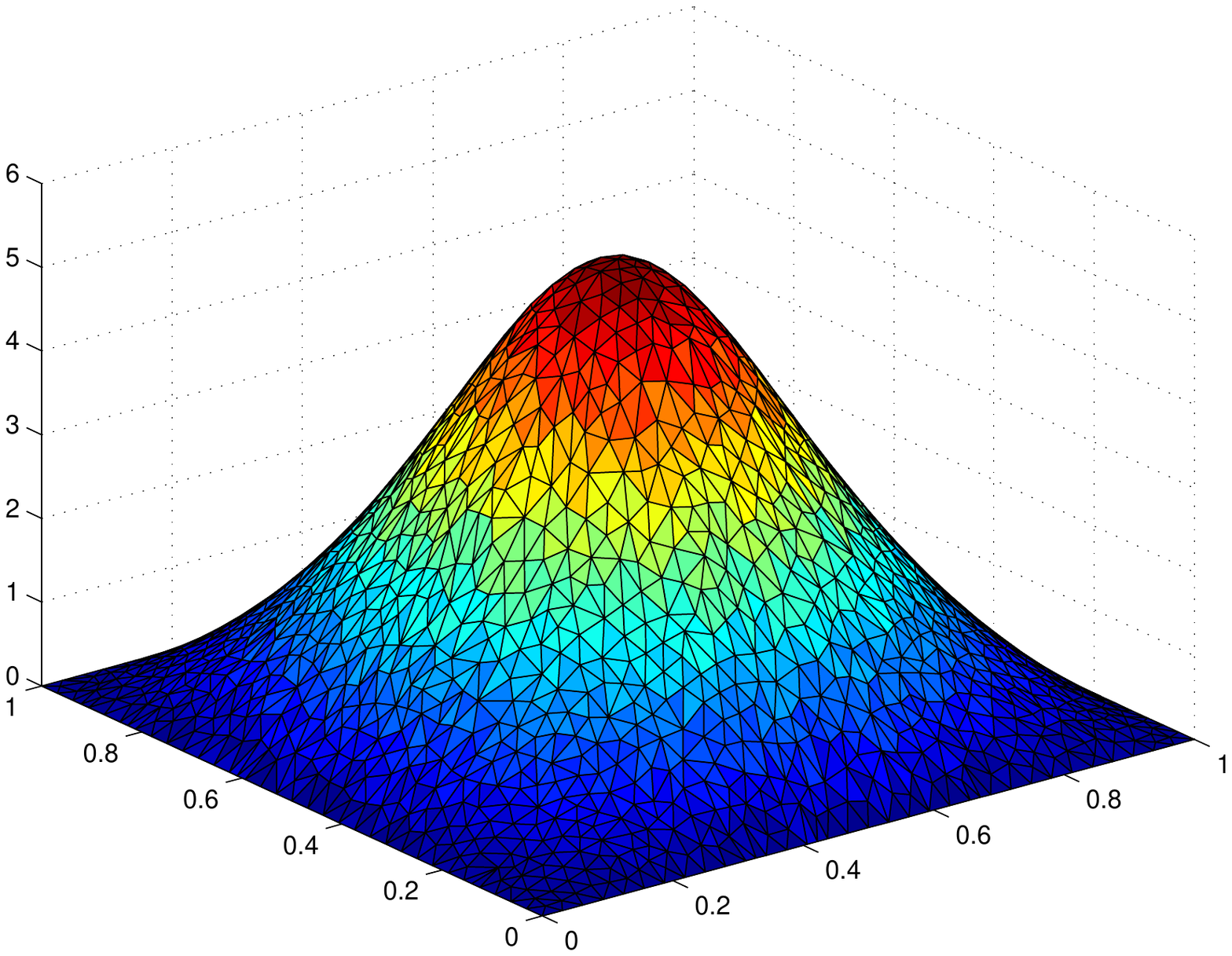}%
  \end{center}
  \vspace{-0.15\linewidth}%
  \caption{MPA solution for the system with
    $(\mu_1, \mu_2, \beta_{1,2}) = (1, 4, -1)$.}
  \label{carre sys -1}
\end{figure}

\begin{figure}[h!t]
  \vspace{-0.13\linewidth}%
  \begin{center}
    \includegraphics[width=0.4\linewidth]{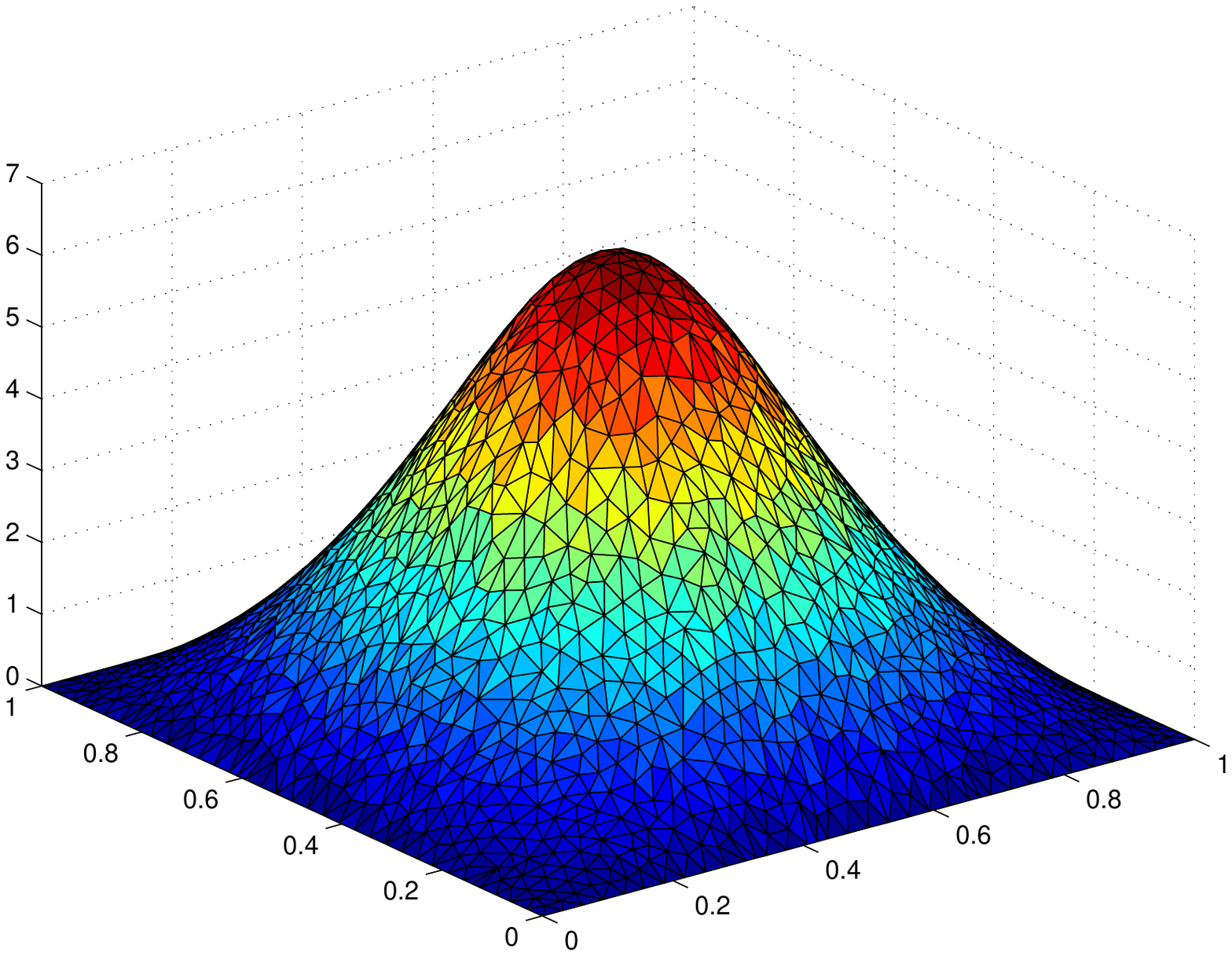}%
    \hfil
    \includegraphics[width=0.4\linewidth]{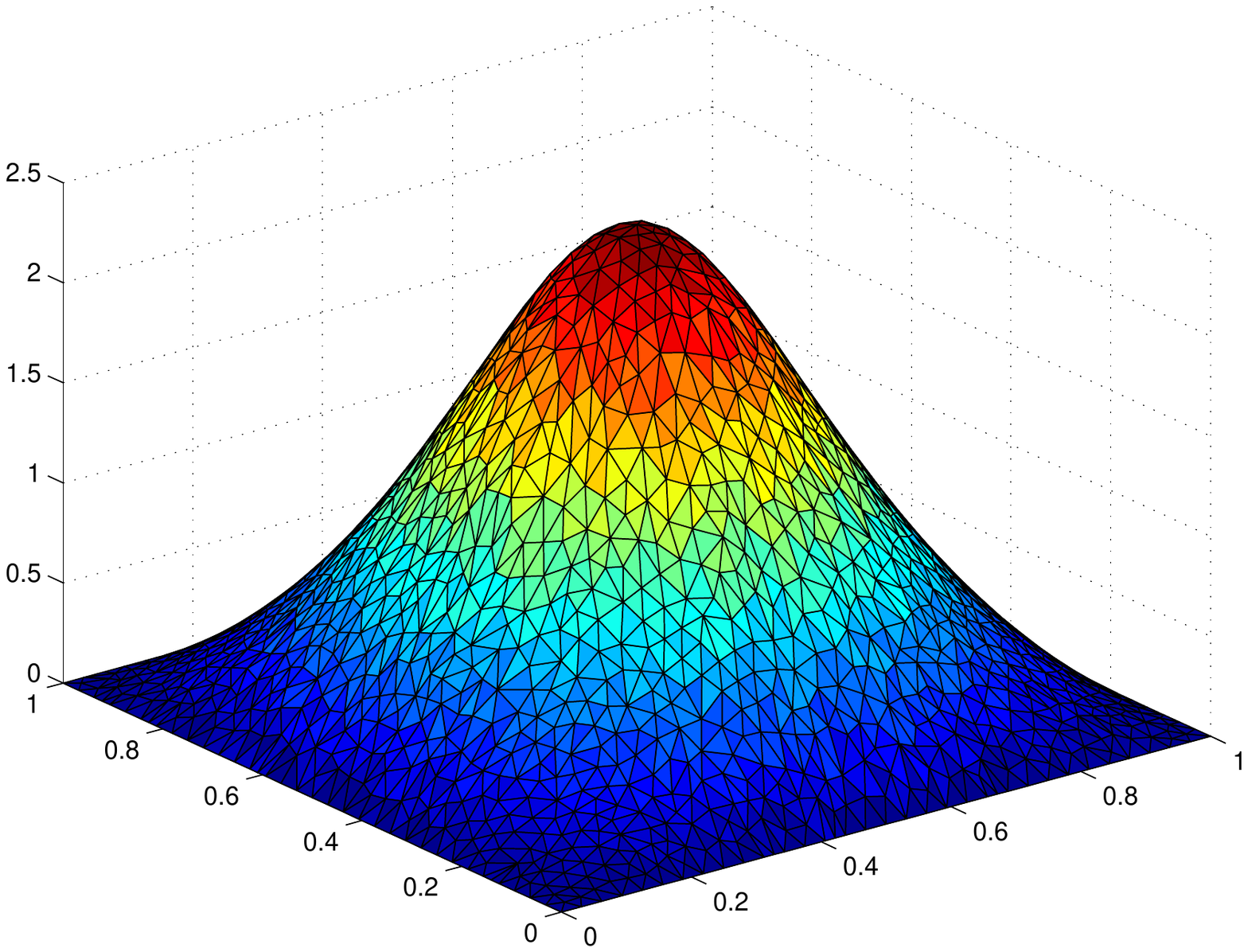}%
  \end{center}
  \vspace{-0.15\linewidth}%
  \caption{MPA solution for the system with
    $(\mu_1, \mu_2, \beta_{1,2}) = (1, 4, 0.5)$.}
  \label{carre sys 0.5}
\end{figure}

\begin{figure}[ht]
  \vspace{-0.13\linewidth}%
  \begin{center}
    \includegraphics[width=0.4\linewidth]{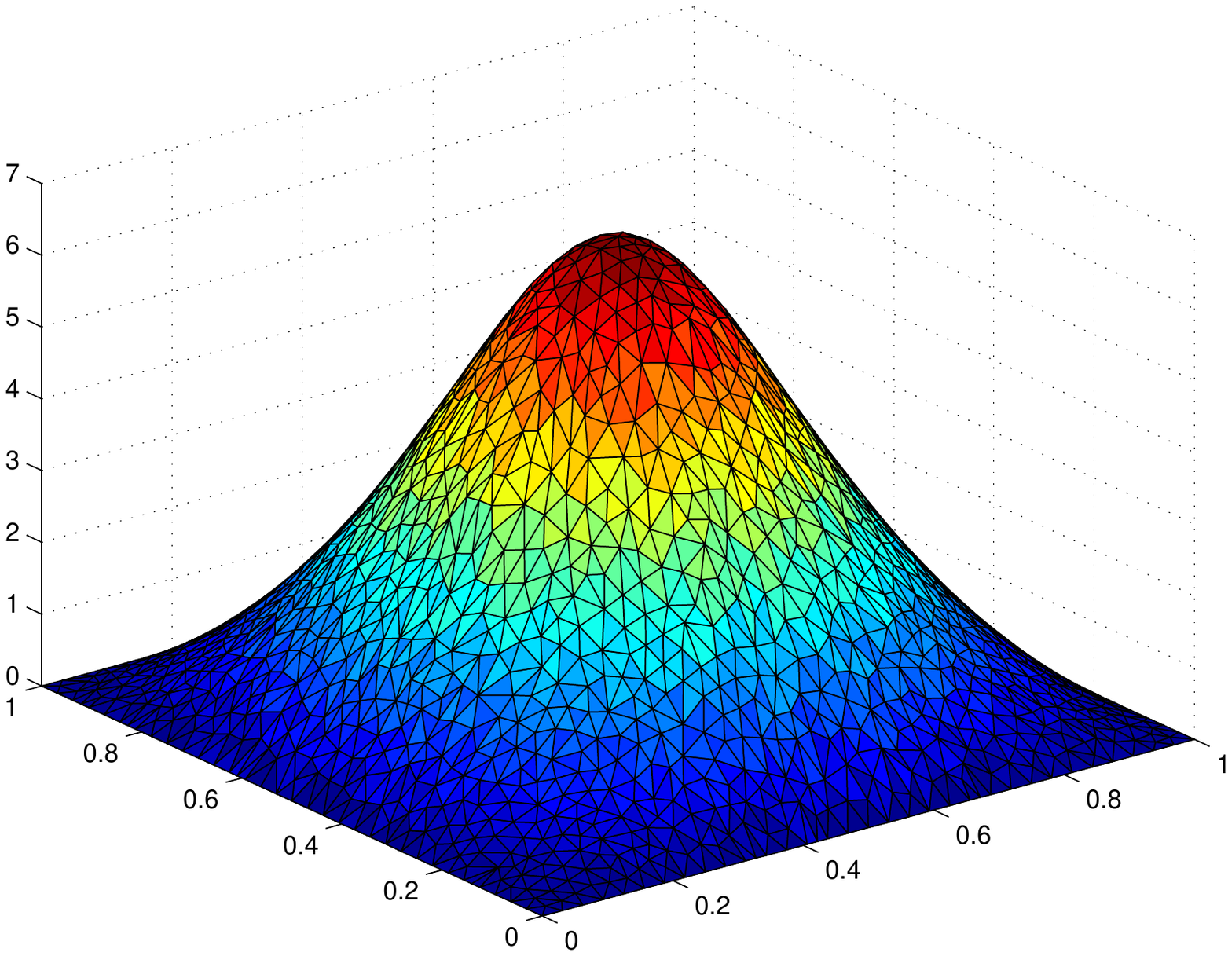}%
    \hfil
    \includegraphics[width=0.4\linewidth]{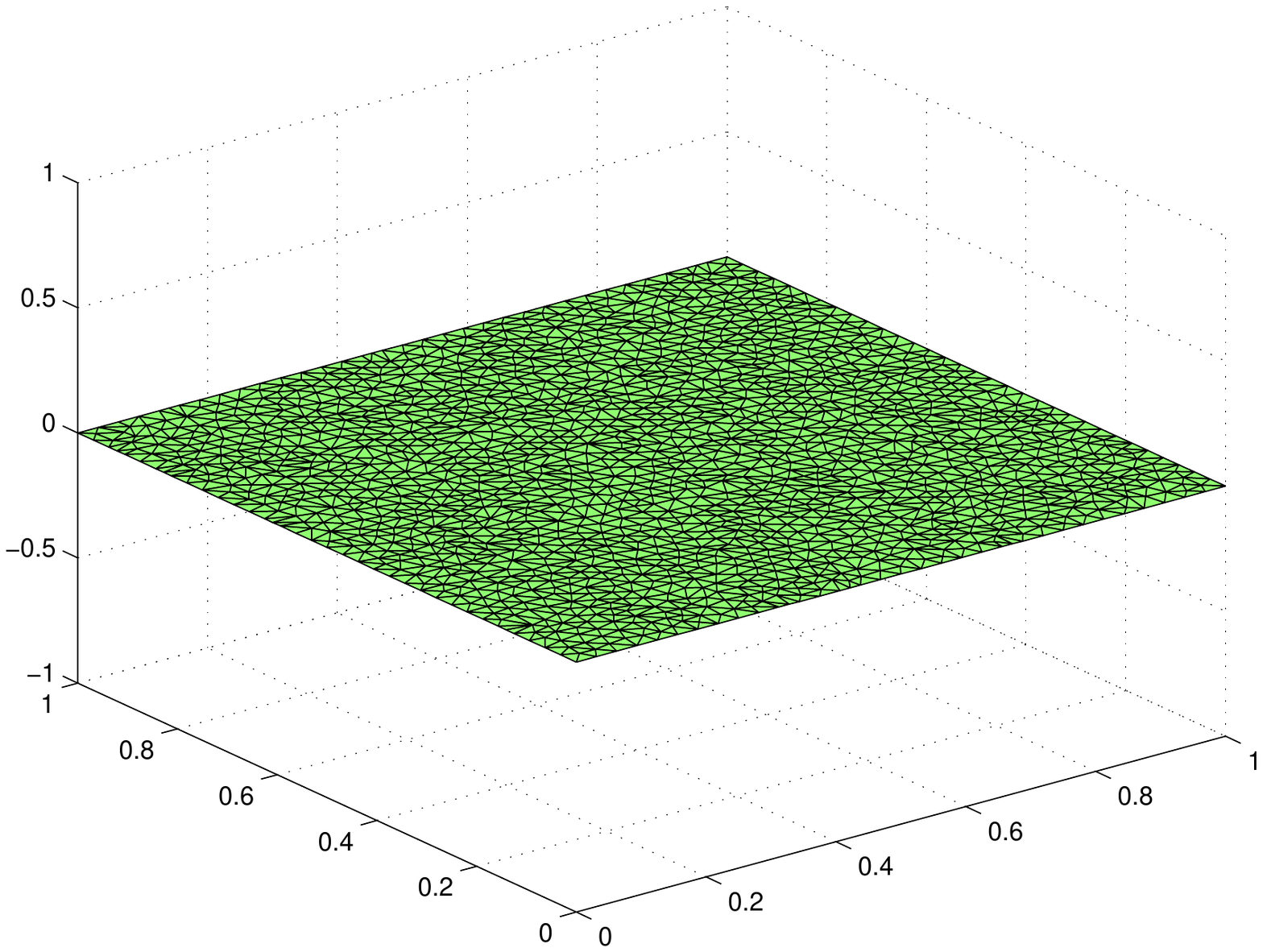}%
  \end{center}
  \vspace{-0.15\linewidth}%
  \caption{MPA solution for the system with
    $(\mu_1, \mu_2, \beta_{1,2}) = (1, 4, 1.2)$.}
  \label{carre sys 1.2}
\end{figure}

\begin{table}[ht]
  \begin{center}
    \begin{tabular}{>{\vrule height 2.3ex depth 0.9ex width 0pt}
        c  r@{.}>{$}l<{$} r r@{.}l r@{.}l r@{.}l}
      $(\mu_1, \mu_2, \beta_{1,2})$&
      \multicolumn{2}{c}{$\norm{\nabla \E(u)}$}
      &\# steps & \multicolumn{2}{c}{$\E(u)$}
      &\multicolumn{2}{c}{$\max u_1$}& \multicolumn{2}{c}{$\max u_2$}\\
      \hline
      $(1, 4, -1)$ & 7&9\cdot 10^{-5} & 11& 88&4&  8&6&  5&4\\ \hline
      $(1, 4, 0.5)$& 5&4\cdot 10^{-5} & 11& 40&4&  6&4&  2&4\\ \hline
      $(1, 4, 1.2)$& 5&2\cdot 10^{-5} & 11& 39&9&  6&6&  0&0
    \end{tabular}
  \end{center}
  \vspace{0.2ex}
  \caption{Characteristics of the solution to system~\eqref{eqNoris}.}
  \label{table-sys}
\end{table}

\bibliographystyle{plain}
\bibliography{MPA}

\end{document}